\documentclass[10pt,reqno]{amsart}
\usepackage{hyperref,amssymb,mathrsfs,graphicx,subfigure}
\usepackage{amsfonts,amssymb,amsmath,enumerate} 
\usepackage{graphicx,colortbl,xcolor}
\usepackage{floatrow}
\usepackage{soul}
\usepackage{comment}

\newtheorem{theorem}{Theorem}[section]
\newtheorem{lemma}{Lemma}[section]

\newtheorem{proposition}{Proposition}[section]

\newtheorem{definition}{Definition}[section]

\DeclareMathOperator*{\essinf}{ess\,inf}
\DeclareMathOperator*{\esssup}{ess\,sup}
\newcommand{\Section}[1]{\section{#1}\setcounter{equation}{0}}

\def\BE{\begin{equation}}
\def\EE{\end{equation}}

\def\MM{\mathsf{M}}

\newcommand{\R}{\mathbb{R}}
\newcommand{\N}{\mathbb{N}}

\newcommand{\eps}{\varepsilon}
\renewcommand{\t}{\tau}

\newcommand{\mm}{\mathtt m} 
\newcommand{\vv}{\mathtt v}

\newcommand{\tv}{\mathrm{TV}\,}

\newcommand{\BV}{$BV$\ }

\newcommand{\DT}{{\Delta t}}

\newcommand\ds{\displaystyle}

\makeatletter
\newcommand{\doublewidetilde}[1]{{%
  \mathpalette\double@widetilde{#1}%
}}
\newcommand{\double@widetilde}[2]{%
  \sbox\z@{$\m@th#1\widetilde{#2}$}%
  \ht\z@=.9\ht\z@
  \widetilde{\box\z@}%
}
\makeatother

\def\charf {\mbox{{\text 1}\kern-.30em {\text l}}}

\begin{document}

\title[]{
Unconditional flocking for Weak solutions to self-organized systems of Euler-type
\\ with all-to-all interaction kernel\\
}

\author[]
{Debora Amadori}
\address{\newline 
Dipartimento di Ingegneria e Scienze dell'Informazione e Matematica (DISIM), University of L'Aquila -- L'Aquila, Italy}
\email{debora.amadori@univaq.it}

\author[]
{Cleopatra Christoforou}
\address{\newline Department of Mathematics and Statistics, University of Cyprus -- Nicosia, Cyprus}
\email{christoforou.cleopatra@ucy.ac.cy}

\begin{abstract}
We consider a hydrodynamic model of flocking-type with all-to-all interaction kernel in one-space dimension and establish that the global entropy weak solutions, constructed in~\cite{AC_2021} to the Cauchy problem for any $BV$ initial data that has finite total mass confined in a bounded interval and initial density uniformly positive therein, admit unconditional time-asymptotic flocking without any further assumptions on the initial data.
In addition, we show that the convergence to a flocking profile occurs exponentially fast.
\end{abstract}
\date{\today}


\thanks{\textbf{Acknowledgment.} The work of D.A. was partially supported by the Ministry of University and Research (MUR), Italy under the grant PRIN 2020 - Project N. 20204NT8W4, \textit{Nonlinear evolution PDEs, fluid dynamics and transport equations: theoretical foundations and applications} and by the INdAM-GNAMPA Project 2023, CUP E53C22001930001, \textit{Equazioni iperboliche e applicazioni}} 

\maketitle
\centerline{\date}

\tableofcontents

\Section{Introduction}%
The study of hydrodynamic models that emerged in the area of self-organization has received alot of attention in the recent years and many new challenges in partial differential equations have arisen that yield interesting questions in the mathematical community. In this paper, we set up a problem in this context and study solutions in the weak framework that are appropriate to capture asymptotic flocking.
To begin with, we consider the Cauchy problem for the system
\begin{equation}\label{eq:system_Eulerian}
\begin{cases}
\partial_t\rho +  \partial_x (\rho \vv)  = 0, &\\
\partial_t (\rho \vv)  +  \partial_x \left(\rho \vv^2 + p(\rho) \right) = K \displaystyle\int_\R \rho(x,t)\rho(x',t)\left(\vv(x',t)  - \vv(x,t) \right)\,dx'  &
\end{cases} 
\end{equation}
with $(x,t)\in \R\times [0,+\infty)$. Here $\rho\ge 0$ stands for the density, $\vv$  for the velocity, $p$ for the pressure, given by
\begin{equation}\label{gamma=1}
p(\rho)= \alpha^2\rho\,, \qquad \alpha>0\,,
\end{equation}
and $K>0$ is a given constant. Having set $\mm := \rho\vv$ as the momentum variable, let the initial condition be
\begin{equation}\label{eq:init-data}
(\rho,\mm)(x,0)=\left(\rho_0(x), \mm_0(x)\right)\,\qquad x\in\mathbb{R}\;,
\end{equation}
and our aim is to formulate a problem to~\eqref{eq:system_Eulerian}--\eqref{eq:init-data} with conditions appropriate for the models of self-organized systems and then seek weak solutions. The pioneering work of Cucker and Smale~\cite{CuS1} led a major part of the mathematical community to conduct research intensively on this topic. Many mathematical models have arised and most work so far is on the behavior of the particle models, the kinetic equation and the hydrodynamic formulation. However, very little is done in this area on \emph{weak solutions} and the scope of this paper is to contribute in this direction of weak solutions to the an Euler-type flocking system that is derived in a hydrodynamic formulation. We refer the reader to the reviews~\cite{CFTV-2010, Shv2021,Ta2023} and the references therein.

Karper, Mellet and Trivisa in~\cite{KMT15} prove the convergence of weak solutions to the kinetic equation Cucker-Smale flocking model on $(x,t,\omega)\in (0,T)\times\R^d\times\R^d$
\begin{equation}\label{S1eq:hydro}
f_t^\epsilon+\omega\cdot\nabla_x f^\epsilon+\text{div}_{\omega}(f^\epsilon L[f^\epsilon])=\frac1\epsilon\Delta_\omega 
f^\epsilon+\frac1\epsilon \text{div}_\omega(f^\epsilon(\omega-\vv^\epsilon))
\end{equation}
to strong solutions of an Euler-type flocking system of the form~\eqref{eq:system_Eulerian}. Here,
$f^\epsilon\dot=f^\epsilon(x,t,\omega)$ stands for the scalar density of individuals and $L$ the alignment operator 
that is the usual Cucker-Smale operator given by
$$ 
L[f] (x,t,\omega)~\dot =~\int_{\R^d}\int_{\R^d} K(x,y) f(y,w)(w-\omega)dw\,dy
$$
where $K$ is a smooth symmetric kernel and $\epsilon>0$ a small positive parameter. Moreover, on the right hand side of \eqref{S1eq:hydro}, the first term is due to the presence of a stochastic forcing at the particle level, see~\cite{BCC-2011} and the last term in~\eqref{S1eq:hydro} is the strong local alignment interaction with $\vv^\epsilon$ to be the average local velocity. 
This alignment term was derived in~\cite{KMT13} from the Motsch-Tadmor~\cite{MT11} alignment operator (MT) as a singular limit and it is known that the MT operator is an improvement of the standard Cucker-Smale model at small scales. In~\cite{KMT15}, they study the singular limit corresponding to strong noise and strong local alignment, i.e. $\epsilon\to0^+$, and show the convergence
\begin{equation}\label{eq:conv-to-Maxwellian}
    f^\epsilon \to f(x,t,\omega)=\rho(x,t) \exp\left({-\frac{|\omega-\vv(x,t)|^2}{2}}\right)\,,
\end{equation}
while the macroscopic variables $\rho,\,\vv$, which are the $\epsilon\to0^+$ limits of
$$
\rho^\epsilon\dot=\int f^\epsilon d\omega,\qquad \rho^\epsilon \vv^\epsilon\dot=\int f^\epsilon \omega d\omega\;,
$$
satisfy the Euler-type flocking system~\eqref{eq:system_Eulerian} with pressure \eqref{gamma=1}. 
However, in the literature, most hydrodynamic models for flocking so far are described by a pressureless Euler system.
Actually, the pressureless systems arise because the microscopic description of the particles motion, considered in those works, does not contain a stochastic forcing, hence the kinetic equation does not contain the diffusion term and the momentum equation can be closed by the mono-kinetic ansatz that renders it different from~\eqref{eq:conv-to-Maxwellian}. 
It is true that the system with pressure received less attention than the pressureless one and especially in  view of weak solutions.
We refer to~\cite{Choi2019} for a result on smooth, space-periodic solutions to this model with pressure and our previous result~\cite{AC_2021} for~\eqref{eq:system_Eulerian} with $K=1$ on weak solutions and conditional flocking. We will elaborate on the results in~\cite{AC_2021} after setting up the problem.

To set up our problem, we assume that the initial mass $\rho_0$ is confined in a bounded interval, and is uniformly positive in there i.e.  there exist $a_0 < b_0$ such that, for $I_0 ~\dot =~[a_0, b_0]$:
\begin{equation}\label{hyp-init_data}
{\rm supp} \{ {(\rho_0,\mm_0)}\}\subset I_0\,, \qquad \essinf_{I_0} \rho_0 >0\,. 
\end{equation}
This assumption is imposed in order our problem to be suitable model for flocking. For this reason, we will also seek solutions having a bounded support for every $t>0$, that is, there exists two continuous curves $t\mapsto a(t)$, $b(t)$, $t\in[0,+\infty)$ with
\begin{equation}\label{cond-on-a-b}
    a(0)=a_0\,,\quad b(0)=b_0\,;\qquad a(t)<b(t)\qquad  \mbox{ for all }t>0
\end{equation}
such that the support satisfies
\begin{equation}\label{eq:bounded-support}
{\rm supp} \{ {(\rho,\mm)(\cdot,t)}\} 
{ ~=~ }
I(t) ~\dot = ~[a(t),b(t)]\,,\qquad t>0\,. 
\end{equation}
Actually, we require that the two \textit{interfaces} $a(t)$ and $b(t)$ correspond to particle trajectories:
\begin{equation}\label{eq:particle-paths}
    a'(t)=\vv(a(t)+,t)\,,\quad b'(t)=\vv(b(t)-,t) \qquad \mbox{ for a.e. }t>0\,.
\end{equation}
Having these, we set 
\begin{equation}\label{def:Omega}
\Omega= \{ (x,t);\ t> 0\,,\  x\in (a(t),b(t))\}\subset \R\times(0,+\infty)
\end{equation} 
as the non-vacuum region and consider the velocity $\vv=\mm/\rho$ on $\Omega$. Then system \eqref{eq:system_Eulerian} can be considered in the sense of distributions on $\Omega$ and rewrites as
\begin{equation}\label{eq:system_Eulerian-rho-mm}
\begin{cases}
\partial_t\rho +  \partial_x \mm  = 0, &\\[2mm]
\partial_t \mm  +  \partial_x \left(\mm^2/\rho + p(\rho) \right) = K \left[\rho(x,t) \displaystyle\int_\R \mm(x',t) \,dx' -  \mm (x,t) \int_\R \rho(x',t) \,dx'\right]\,.
\end{cases} 
\end{equation}
We aim for a notion of \emph{entropy weak solutions with concentration} that is motivated by the \textit{ad-hoc} boundary condition:
\begin{equation*}
\text{\emph{The vacuum region is connected with the non-vacuum one by a shock discontinuity}\,.} 
\end{equation*}
This motivation has as a target to capture a sharp front with finite speed around the non-vacuum region $\Omega$ that is expected to arise in flocking. In this way, we exclude the case of a rarefaction connecting a vacuum region with a non-vacuum since in such a case, due to the pressure law \eqref{gamma=1}, the front would not have a proper interpretation in terms of flocking because of the unbounded maximal speed, 
having that $\int_0^{\rho} \frac{\sqrt{p'(s)}}{s}\,ds = +\infty$ for $\rho>0$.

The definition of the \emph{entropy weak solution with concentration} is stated in Section~\ref{S2} and in short, we can say that such  solution is entropy weak to~\eqref{eq:system_Eulerian-rho-mm}  in $\Omega$ but delta point masses are present along the interfaces in the momentum equation when the solution is tested on the whole half plane. The point masses correspond to the values of pressure evaluated within the non-vacuum region and render the solution conserving \emph{mass} and \emph{momentum} in the following sense: From the integral identity of~\eqref{eq:system_Eulerian-rho-mm}$_1$, that is
\begin{equation*}
  \iint\left\{\rho\phi_t + \mm\phi_x \right\}\; dxdt=0\;,\qquad \forall\, \phi\in C^\infty_0(\R\times (0,\infty))
   \end{equation*}
and the Rankine-Hugoniot condition $[\mm] = \dot x [\rho]$ that hold along the interfaces $a(t)$ and $b(t)$ thanks to \eqref{eq:particle-paths},
we get immediately that conservation of \emph{total mass} holds true: 
\begin{equation}\label{cons-of-mass}
\int_\R \rho(x,t)\,dx = \int_{I(t)} \rho(x,t)\,dx = \int_\R \rho_0(x)\,dx =: \MM\,,\qquad\forall\, t\ge 0\,. 
\end{equation}
Here $[\cdot]$ denotes as usual the jump of the values. In addition, the concept of \textit{total extended momentum} defined as
\begin{align}\label{def-M-M1-t}
M_1(t) := \int_{I(t)}   \mm(x,t)dx +P_b(t)-P_a(t)\;,
\end{align}
where
\begin{equation}\label{def:Pnu-intro}
\begin{cases}
\ds P_b(t) :=  \int_0^t e^{-{ K}\MM(t-s)}p(\rho(b(s)-,s))\, ds\,,  & \\[2mm]
\ds P_a(t) :=  \int_0^t e^{-{ K}\MM(t-s)} p(\rho(a(s)+,s))\, ds\,;
    \end{cases}
\end{equation}
would be shown to be conserved as well from our notion of solution (see Prop.~\ref{prop:cons-mass-momentum}). 

As it is also discussed in \cite{AC_2021}, the occurrence of a singularity for the solution in the sense that the momentum exhibits concentration at the interfaces $a(t)$ and $b(t)$ as Dirac deltas leads to the definition of \emph{the extended momentum} $\widehat\mm$ as the following distribution:
\begin{equation}\label{def:m-hat}
    \widehat \mm(\cdot,t) := \mm (\cdot,t) + \delta_{b(t)} P_b(t) - \delta_{a(t)} P_a(t)\,,\quad t>0\,,
\end{equation}
where $\delta_{x_0}$ denotes the Dirac delta function at $x_0 \in \R$. This new singularity of the extended momentum $\widehat m$ along the interfaces $a(t)$ and $b(t)$ is known as delta shock and references can be found in~\cite{Dafermosbook}, Chapter 9.

We further use the following standard notation $<\cdot,\cdot>$:
$$
<\widehat\mm(\cdot,t),\phi(\cdot,t)>:=\int_{I(t)} \mm (x,t)\phi(x,t)dx+ P_b(t)\phi(b(t),t) -  P_a(t)\phi(a(t),t),\quad t>0 
$$
as the value of the functional $\widehat \mm$ over $C_0^\infty$, for all test functions $\phi\in C_0^\infty(\R\times\R_+)$. Note here that $\mm=0$ for $x\notin I(t)$.
It is shown in Proposition~\ref{prop:cons-mass-momentum}  that solutions in the sense of Definition~\ref{entropy-sol} satisfy the conservation of total mass \eqref{cons-of-mass} and  total extended momentum:

\begin{equation}\label{cons-of-momentum}
\int_{I(t)} \mm(x,t) \,dx + P_b(t) - P_a(t) = \int_\R \mm_0(x) \,dx=: \MM_1\,,\quad\forall\, t\ge 0  \,,
\end{equation}
and at this point it is evident that concentration terms in the total momentum are the appropriate fit for the set up of this problem.

The aim of this paper is to show \emph{unconditional time asymptotic flocking} for entropy weak solutions with concentration ~\eqref{eq:system_Eulerian}, \eqref{eq:init-data} with \eqref{gamma=1} that is a significant improvement of our previous result in~\cite{AC_2021}, in which a condition  associated with the initial data is imposed.

Before we state the main result, we set the average velocity to be
\begin{equation}\label{def:vbar}
\bar \vv ~\dot =~ \frac{\MM_1}{\MM}\,,
\end{equation}
and provide the definition of \emph{time-asymptotic flocking}.

\begin{definition}\label{Def: flo}
A solution $(\rho,\mm)(x,t)$ to system~\eqref{eq:system_Eulerian} with the structure~\eqref{eq:bounded-support} admits \emph{time-asymptotic flocking} if the following conditions hold true:
\begin{enumerate}
\item[(i)] the length of the support $I(t) = [a(t), b(t)]$ of the solution is bounded uniformly in  time, 
\item[(ii)] the oscillation of the velocity $\vv=\mm/\rho$  on $(a(t), b(t))$ vanishes as $t\to+\infty$, that is
$$\esssup_{x_1,x_2\in (a(t), b(t))} |\vv(x_1,t) - \vv(x_2,t)|~~\xrightarrow{t\to+\infty} ~~ 0\;.
$$
\end{enumerate}
\end{definition}

Here we state the main theorem of this paper.
\begin{theorem}\label{Th-2-unconditional}  Assume that the initial data $(\rho_0,\mm_0)\in BV(\R)$ satisfy \eqref{hyp-init_data}.
Then there exists an entropy weak solution $(\rho,\mm)$ with concentration  along $a(t)$ and $b(t)$ to the Cauchy problem \eqref{eq:system_Eulerian}, \eqref{eq:init-data} with pressure~\eqref{gamma=1} that satisfies \eqref{cons-of-mass} and \eqref{cons-of-momentum} and admits \emph{time-asymptotic flocking}.
In particular, the decay occurs exponentially fast and for  some $\rho_\infty>0$,
one has that
\begin{equation}\label{Th-2exp}
    \esssup_{x\in (a(t), b(t))} |\rho(x,t)- \rho_\infty|\,,\qquad \esssup_{x\in (a(t), b(t))} |\vv(x,t) - \bar \vv|
\le C_2'e^{- C_1'\sqrt{K}   t},\,\,\qquad \forall\, t>0\,
\end{equation}
for some positive constants $C_1',\,C_2'$.
\end{theorem}

\medskip
Global existence of such weak solutions is shown in~\cite{AC_2021} for $K=1$ and without much effort this is extended to any constant kernel $K>0$ in our setting  by scaling the variables. However, the time asymptotic flocking needs careful treatment different from~\cite{AC_2021}. Let us point out that in our previous result~\cite{AC_2021}, time asymptotic flocking is achieved under a special condition on the bulk of the initial data. More precisely, by means of Theorem~1.2 in \cite{AC_2021} and the case $K=1$, one obtains a sufficient condition for time-asymptotic flocking that involves the parameter $\alpha$ in~\eqref{gamma=1}, the total mass $\MM$, pointwise values of $\rho_0$ at the endpoints $a_0+$ and $b_0-$ and the quantity
\begin{equation}\label{eq:def-q}
q: = \frac 12 \tv \{\ln(\rho_0)\} +  \frac 1{2\alpha} \tv \{\vv_0\}    \,.
\end{equation}
Moreover, the oscillation of the velocity decays exponentially fast to zero. In the present work, we prove that the time-asymptotic flocking occurs \emph{unconditionally}, i.e. for \emph{any} initial data $(\rho_0,\mm_0)\in BV(\R)$ that satisfy \eqref{hyp-init_data} without any further condition and also capture exponential decay again. Furthermore, it is immediate from our analysis that the density $\rho$ decays exponentially fast to a constant state as obtained in~\eqref{Th-2exp} while this is not a requirement for flocking according to Definition~\ref{Def: flo}.

Our analysis relies alot on the construction in~\cite{AC_2021} of an approximate solutions that converge up to a subsequence in $L^1_{loc}$ to \emph{entropy weak solutions with concentration} along the interfaces $a(t)$ and $b(t)$, but it requires different tools and strategy to study the wave decay and control the total variation as $t\to\infty$ without any further conditions. Let's introduce the construction in ~\cite{AC_2021}.
First, by conservation of total mass and total extended momentum, the system can be reduced to a local one.
Indeed, by means of \eqref{cons-of-mass} and \eqref{cons-of-momentum}, the integral term on the right hand side of \eqref{eq:system_Eulerian} can be rewritten as
\begin{align*}
K\rho(x,t) \left\{ <  \widehat\mm(\cdot,t) ,\phi_1> -  \vv(x,t)  \int_\R \rho(x',t)\,dx'\right\}
&=  K\rho(x,t) \left(\MM_1 - \vv(x,t) \MM \right)\\
&= K \MM \rho(x,t) \left(\bar \vv - \vv(x,t)\right)\,,
\end{align*}
where $\phi_1$ is a test function equal to 1 for $x\in I(t)$ while $\bar\vv$ is the average speed given at~\eqref{def:vbar}.
Thus, system \eqref{eq:system_Eulerian} or~\eqref{eq:system_Eulerian-rho-mm} rewrites as
\begin{equation}\label{eq:system_Eulerian_M-M1-K}
\begin{cases}
\partial_t\rho +  \partial_x \mm  = 0, &\\[2mm]
\partial_t \mm  +  \partial_x \left(\mm^2/\rho + p(\rho) \right) = - K\, \MM \rho \left(\vv -\bar \vv \right)\,.
\end{cases} 
\end{equation}
and the equivalence between the nonlocal and the local system relies on the conservation of $\MM$ and $\MM_1$, hence it is suitable in the setting of our notion of weak solution with concentration. 

Now, to construct a convergent approximate sequence to systems of conservation laws, we require strict hyperbolicity. 
Since the Cauchy problem~\eqref{eq:system_Eulerian_M-M1-K},~\eqref{eq:init-data} is not strictly hyperbolic  on $\mathbb{R}\times(0,\infty)$, we transform the  problem into Lagrangian coordinates in the spirit of Wagner~\cite{W87}. However, the equivalence from Eulerian $(\rho(x,t),\mm(x,t))$ into the Lagrangian variables $(u(y,t), v(y,t))$ needs to be verified in our set up because in~\cite{W87}, it is assumed infinite total mass. Actually, our finite total mass condition has as an implication that the problem in Lagrangian coordinates is an initial boundary problem with fixed boundaries at $y=0$ and $y=\MM$. More precisely, by recasting system~\eqref{eq:system_Eulerian_M-M1-K} from Eulerian $(\rho(x,t),\mm(x,t))$ into the Lagrangian variables $(u(y,t), v(y,t))$, we obtain the equations
\begin{equation}\label{S1eq:system_Lagrangian}
\begin{cases}
\partial_\t u -  \partial_y  v  = 0, &\\
\partial_\t v  +  \partial_y (\alpha^2/u) =   - K\MM (v-\bar \vv) & 
\\ 
\end{cases} 
\end{equation}
with the domain $\{(y,t);\ t\ge 0\,,\ y\in (0,\MM)\}$\,. 
It should be mentioned that global existence of weak solutions to the homogeneous system ($\MM=0$) corresponding to~\eqref{S1eq:system_Lagrangian} was first obtained by Nishida~\cite{Nishida68} using the random choice method for initial data of large \BV for both the Cauchy and the boundary case. For the non-homogeneous system~\eqref{S1eq:system_Lagrangian}, global weak solutions were constructed in~\cite{Dafermos_frictional,LuoNatYan,AmadoriGuerra01}, using either the random choice or the front tracking methods, but only for the Cauchy problem. 
Also, Frid~\cite{Frid96} studied existence under certain initial-boundary value problems on a bounded domain for systems arising in isentropic gas dynamics and elasticity theory, which are different type from ours. Therefore, the aformentioned results do not apply in our setting and for this reason the machinery of front tracking algorithm was incorporated to establish the existence in~\cite{AC_2021} for $K=1$. This result is trivially extended to any constant $K>0$. 

The heart of the matter is to prove that the total variation decays to zero as time tends to infinity without any further condition on the initial data. This requires delicate estimates on the cancellation properties of the wave strengths, that leads to wave decay. 
The analysis is performed around interaction of waves and time steps at the level of the approximate solutions using a novel approach to quantify the change of the total variation and get sharp bounds. Our approach relies on the interplay between the rarefaction part and the total variation and a chase to reveal the hidden relation between the two. The crucial point that uncovers the vanishing property of the total variation is that the rarefaction part of the solution is controlled by the variation of the linear functional. Since the linear functional is non-increasing, if we wait long enough, the  rarefaction part can become arbitrarily small. The key idea to show that the rarefaction part can be controlled in this way relies on a region that serves as a trapped area for each family of rarefaction waves (see Lemma~\ref{lem:rar-vanishes-for-large-t} and Figure~\ref{Fig1}). This region is actually triangular in Lagrangian coordinates and with an appropriate weighted functional we study the possible cases that can occur. As a consequence, also the total amount of shocks must vanish as $t\to\infty$: indeed, 
we show in Lemma~\ref{L4.2} that a persisting positive amount of shocks would  produce a uniformly positive amount of rarefactions, which contradicts the property that rarefactions vanish as $t\to\infty$. Finally, again by means of wave decay, we show in Lemma~\ref{S5Prop 5.1} that the total variation cannot stay uniformly positive for arbitrarily large times and then that there exists a time $T^*$ such that the total variation of the unknowns $\rho$, $\vv$ reduces by a geometric rate when passing from a time $\tau$ to time $\tau+T^*$ within the non-vacuum region $\Omega$.

The paper is organized as follows: We first state the definition of \emph{entropy weak solution with concentration} in Section~\ref{S2} and show conservation of total mass and total extended momentum for such solutions. Then in Section~\ref{S3}, we describe the strategy of our proof that is postponed for the next sections. Actually, in Section~\ref{S2.1}, we show that the solution admits time asymptotic flocking and this is proven to happen exponentially fast in Section~\ref{S5}.

\Section{Setting of the problem}\label{S2} 
In this section, we first present the definition of \emph{entropy weak solution with concentration} and then, we exploit it in terms of the conserved quantities. As it is shown, the concentration terms in~\eqref{def:m-hat} are crucial to establish conservation of total momentum, in the sense of \eqref{cons-of-momentum}.

First, we consider the functions 
$$
\eta(\rho,\mm)\,,\quad q(\rho,\mm)\,,
$$ 
defined on $(0,+\infty)\times \R$, in terms of $\rho> 0$ and $\mm$, that constitute a pair of entropy-entropy flux functions for the system \eqref{eq:system_Eulerian} i.e. they are differentiable on $(0,+\infty)\times \R$, $\eta$ is convex and the following relations hold
\begin{equation*}
    \left(-\vv^2 + p'\right)\eta_{\mm} = q_\rho\,,\qquad 
    \eta_\rho + 2\vv \eta_\mm = q_{\mm}\,.
\end{equation*}

Let us now introduce the appropriate notion of entropy weak solution for \eqref{eq:system_Eulerian}, or equivalently \eqref{eq:system_Eulerian-rho-mm}, 
together with \eqref{eq:init-data} and \eqref{gamma=1}, which we call \emph{entropy weak solution with concentration}. 

\begin{definition}\label{entropy-sol} 
Given $(\rho_0,\mm_0)\in BV(\R)$ that satisfy \eqref{hyp-init_data}, let $(\rho,\mm): [0,+\infty)\times \R \to \R^2$ be a function with the following properties:
\begin{itemize}
    \item the map $t\mapsto (\rho,\mm)(\cdot,t) \in L^1_{loc} \cap~ BV$  is continuous in $L^1_{loc}$;
    \smallskip
 \item 
 $\displaystyle\lim_{t\to 0+}(\rho, \mm)(\cdot,t)=\left(\rho_0, \mm_0\right)$ in $L^1_{loc}$;
    
     \smallskip
     
    \item there exist two locally Lipschitz curves $a(t)$ and $b(t)$, $t\in[0,+\infty)$, and a value $\rho_{inf}>0$ such that \eqref{cond-on-a-b}, \eqref{eq:particle-paths}, \eqref{eq:bounded-support} hold and  
\begin{equation}\label{eq:density-unif-positive}
\essinf_{I(t)} \rho(\cdot,t) \ge \rho_{inf}>0\qquad \forall\, t>0\,.
\end{equation}
\end{itemize}
Then $(\rho, \mm)$ is an entropy weak solution \textit{with concentration} along $a(t)$ and $b(t)$ of the problem~\eqref{eq:system_Eulerian},
\eqref{eq:init-data} with \eqref{gamma=1}, if 
\begin{itemize}
    \item[(a)]
 the following integral identities hold true for all test functions $\phi\in C^\infty_0(\R\times (0,\infty))$:
\begin{align}
&\iint \left\{\rho\phi_t + \mm\phi_x \right\}\; dxdt=0\;,\label{S1:rho-eq-phi2}\\
&\iint\left\{ \mm\phi_t
    + \left[ \frac{\mm^2}{\rho}+ p(\rho)  \right]\phi_x \right\}dx dt - { K} \iint\left[\MM\, \mm
    -M_1(t)\, \rho 
     \right] \phi \,dx dt \nonumber\\
&  -  \int_0^\infty   \left[p(\rho(b(t)-,t)) \phi(b(t),t) -p(\rho(a(t)+,t)) \phi(a(t),t) \right] \,dt=0  \label{S1:m-eq-phi2-pressure}
\end{align}
where $\MM$ and $M_1(t)$ are given in \eqref{cons-of-mass} and~\eqref{def-M-M1-t}, respectively.
\item[(b)] For every entropy-entropy flux pair $
\eta(\rho,\mm)$\,, $q(\rho,\mm)$
defined on $(0,+\infty)\times \R$ and with $\eta$ convex,
the inequality
 \begin{equation}\label{entropy-cond_rho-m}
    \partial_t \eta(\rho,\mm) + \partial_x q(\rho,\mm)\le { K}\, \eta_{\mm}  \left[M_1(t)\,\rho  - \MM\,\mm \right]
\end{equation}
holds in ${\mathcal D}'(\Omega)$.
\end{itemize}
\end{definition}

From Definition~\ref{entropy-sol}, taking test functions supported in $\Omega$, it is immediate  that this notion of solution considered is an entropy weak solution to the system
\begin{equation}\label{S2_eq:system_Eulerian - a}
\begin{cases}
\partial_t\rho +  \partial_x \mm  = 0, &\\
\ds \partial_t \mm  +  \partial_x \left(\frac{\mm^2}{\rho} + p(\rho) \right) = K \rho M_1(t) -
K \mm(x,t) \int_{I(t)} \rho(x',t) \,dx'  &
\end{cases} 
\end{equation}
for $(x,t) \in \Omega$  rather than $\R\times \R_+$, where $M_1(t)$ is given at~\eqref{def-M-M1-t}. However, for test functions continuous up to the interfaces $a(t)$ and $b(t)$, it is necessary to take into account boundary terms that arise from the term $\mm^2/\rho$. The presence of these boundary contributions in the notion of solution was motivated in~\cite[Sec. 2.1]{AC_2021} by studying the Riemann problem around vacuum. More precisely, consider the homogeneous system~\eqref{eq:system_Eulerian-rho-mm}, i.e. 
\begin{equation}\label{S2_eq:system_Eulerian - hom}
\begin{cases}
\partial_t\rho +  \partial_x \mm = 0, &\\
\ds \partial_t \mm  +  \partial_x \left(\frac{\mm^2}{\rho} + p(\rho) \right) = 0&
\end{cases} 
\end{equation}
with Riemann data
\begin{equation}\label{RH1-dataapprox}
(\rho,\mm)(x,0)=\left\{\begin{array}{ll}
(\rho_\ell,\mm_\ell) & x<0 
\\
(\delta,\mm_r(\delta))& x>0\,,
\end{array}\right.
\end{equation}
denote by $(\rho_\delta,\mm_\delta)$ its Riemann solution that is a $2$-shock and then take the limit $\delta\to0+$ to capture the $2$-shock solution $(\rho,\mm)$ around vacuum at the point $x=b_0=0$. Here the left state $(\rho_\ell,\mm_\ell)$ corresponds to a fixed state with $\rho_\ell>0$, while the right state $(\delta,\mm_r(\delta))$, where $\mm_r(\delta)=\delta\,\vv_r(\delta)$, varies in terms of the parameter $\delta>0$.
Using the Rankine-Hugoniot conditions 
\begin{align}\label{RH1}
    \sigma(\rho_\ell-\delta) &=\rho_\ell \vv_\ell-\delta\, \vv_r(\delta)\\[2mm] \label{RH2}
    \sigma(\rho_\ell \vv_\ell-\delta \vv_r(\delta)) &=\rho_\ell \vv_\ell^2-\delta\, \vv_r(\delta)^2+p(\rho_\ell)-p(\delta) 
\end{align}
and the shock wave curve of the second family
\begin{equation}\label{S2reply} 
 S_2:\quad  \vv_r(\delta) =  
    \vv_\ell -\sqrt{\dfrac{\left(p(\delta)-p(\rho_\ell)\right) \left(\delta - \rho_\ell \right) }{\delta\rho_\ell}
		}  \qquad  0<\delta\le\rho_\ell\,, 
\end{equation} 
we can determine this solution
\begin{equation}\label{def:RP-sol-rho-m}
(\rho_\delta,\mm_\delta)(x,t)=\left\{\begin{array}{ll}
(\rho_\ell,\mm_\ell
)\,, & \quad x< \sigma t \\
(\delta, \mm_r(\delta))\,,&\quad  x>\sigma t
\end{array}
\right.
\end{equation}
for $\delta \in (0,\rho_\ell)$ where $\mm_r(\delta) = \delta \, \vv_r(\delta)$ is obtained from \eqref{S2reply} and the shock speed $\sigma$ is given by \eqref{RH1}, \eqref{RH2} noting that $\sigma$ also depends on $\delta$. As $\delta\to0+$, we have $\sigma\to \vv_\ell$ and $(\rho_\delta,\mm_\delta)\to(\rho,\mm)$ and it is immediate that the limit satisfies
 \begin{equation}\label{def:RP-sol-tilde-rho-m}
(\rho,\mm)(x,t):=\left\{\begin{array}{ll}
(\rho_\ell,\mm_\ell
)\,, &\quad x< t \vv_\ell  \\
(0,0)\,,&\quad x>t \vv_\ell\,.
\end{array}
\right.
\end{equation}
Now, the weak formulation of the homogeneous system~\eqref{eq:system_Eulerian-rho-mm}$_2$, for $\delta>0$ reads as
\begin{equation*}
\iint \left\{\mm_\delta\,\phi_t + \left(\frac{(\mm_\delta)^2}{\rho_\delta} + p(\rho_\delta) \right)\phi_x \right\}\; dxdt=0
\end{equation*}
for all $\phi \in C_0^\infty(\R\times (0,+\infty))$ and as $\delta\to 0+$, one can check that the term $(\mm_\delta)^2/\rho_\delta$ in the region $x > \sigma t$ does not vanish, but it gives a non trivial contribution 
\begin{equation*}
    \iint_{\{x> \sigma t \}} \frac{(\mm_\delta)^2}{\rho_\delta}\phi_x \; dxdt 
    ~~\xrightarrow{\delta\to0+}~~ - p(\rho_\ell) \int_0^{+\infty} \phi( \vv_\ell t,t)\;dt \,,
\end{equation*}
along the interface $x=\vv_\ell t$.
In fact, this results into a concentration term $p(\rho_\ell)$ centered at $x=b(t)=\vv_\ell\, t$.
Therefore, the limit function $(\rho, \mm)$ satisfies the integral identities \eqref{S1:rho-eq-phi2} and
\begin{equation*}
 \int_0^{+\infty} \left\{\int_{\R}\mm \,\phi_t \right.  +  \left. p(\rho)\phi_x \;dx + 
    \int_{x< t \vv_\ell}\left(\frac{\mm^2}{\rho}\right)\phi_x\;dx -p(\rho_\ell)\phi(\vv_\ell t,t) \right\}\; dt=0\,.
\end{equation*}
We point out that the above computation is performed for the homogeneous system~\eqref{S2_eq:system_Eulerian - hom} and one can reach the integral identity~\eqref{S1:m-eq-phi2-pressure} if the source term is included and both endpoints of the support $(a(t),b(t))$ are involved. Another point is that one could express the weak formulation for the extended momentum $\hat m$ using a \emph{Dirac delta} that is concentrated 
along the discontinuity $\gamma=\{(\vv_\ell t,t);\ t\ge 0\}$.

 Regarding now the conserved quantities, it is obvious that the conservation of total mass as mentioned in~\eqref{cons-of-mass} holds. 
 Since there are concentration terms through the values of the pressure along the two interfaces, we confirm in the next proposition that the extended momentum $\hat\mm$ given in~\eqref{def:m-hat} yields conservation of total momentum. It is of course important that the absolutely continuous functions $a(t) < b(t)$, which are also unknowns of the problem, correspond to particle paths, i.e.~\eqref{eq:particle-paths},
and as it is shown below from here and on $M_1(t)$ is constant in time equal to $\MM_1$, i.e. \eqref{cons-of-momentum}. Thus, the equivalence of the local system~\eqref{eq:system_Eulerian_M-M1-K} with the nonlocal~\eqref{eq:system_Eulerian} or~\eqref{eq:system_Eulerian-rho-mm} is verified for \emph{entropy weak solutions with concentration}.

\begin{proposition}\label{prop:cons-mass-momentum} 
An entropy weak solution \textit{with concentration} along $a(t)$ and $b(t)$ to the problem~\eqref{eq:system_Eulerian}, \eqref{eq:init-data} with \eqref{gamma=1} according to Definition~\ref{entropy-sol} conserves total mass and total extended momentum, i.e. it satisfies \eqref{cons-of-mass} and \eqref{cons-of-momentum}.
\end{proposition}

\begin{proof} The conservation of mass \eqref{cons-of-mass} follows from the Rankine-Hugoniot conditions that hold true along the interfaces by~\eqref{eq:particle-paths}.

Concerning \eqref{cons-of-momentum}, we prove that the choice \eqref{def:Pnu-intro} of the point masses at $a(t)$, $b(t)$ yield the conservation of the total extended momentum.
Indeed, by definition~\eqref{def:Pnu-intro}, we have 
\begin{equation*}
    P_b'(t) +  K \MM P_b(t) = p(\rho(b(t)-,t))\,,
\end{equation*}
and therefore
\begin{align*}
    \int_0^{\infty}<\delta_{b(t)} P_b(t),\phi_t(\cdot,t) - K\MM \phi> dt = - \int_0^\infty   p(\rho(b(t)-,t)) \phi(b(t),t) \,dt \,.
\end{align*} 
An analogue identity holds at $x=a(t)$\, and therefore, by definition \eqref{def:m-hat}, we are led to the identity:
\begin{align*} 
&\int_0^{\infty}<\widehat\mm(\cdot,t),\phi_t(\cdot,t) - K\MM \phi> dt = \iint_{\Omega} \mm \left(\phi_t -K \MM\phi \right)\,dxdt \\
&\qquad - \int_0^\infty   \left[p(\rho(b(t)-,t)) \phi(b(t),t) -p(\rho(a(t)+,t)) \phi(a(t),t) \right] \,dt \,. 
\end{align*}
Hence, the integral identity~\eqref{S1:m-eq-phi2-pressure} reduces to:
\begin{align}\nonumber
& \int_0^{\infty}<\widehat\mm(\cdot,t),\phi_t(\cdot,t) > dt    + \iint  \left[ \frac{\mm^2}{\rho}+ p(\rho)  \right]\phi_x dx dt \\ 
& \qquad \qquad + K \iint     \rho \,{ M_1(t)} 
     \phi \,dx dt   -  K \MM \int_0^{\infty}<\widehat\mm(\cdot,t),\phi> dt 
     =0\;. \label{eq:id-for-m-hat}
\end{align}
If we choose the test function $\phi(x,t) = \phi_1(x)\psi(t)$, with $\phi_1(x) =1$ for all $x\in \cup_{t\in [T_1,T_2]} I(t)$ and $\psi(t)=0$ for $t\not \in [T_1,T_2]$, with $0<T_1<T_2$, and notice that $M_1(t)=<  \widehat\mm(\cdot,t) ,\phi_1>$, we get $M_1(T_2)-M_1(T_1)=0$. 

Since this holds true for arbitrary times, $0<T_1<T_2$, we deduce the conservation of the total extended momentum $M_1(t)$.
By the time continuity of $\int \mm(\cdot,t)\, dx$ and the definition \eqref{def:Pnu-intro}, we conclude that $M_1(t) = M_1(0+)$ that is  \eqref{cons-of-momentum}.
\end{proof}

\Section{Outline of the proof}\label{S3}
In this section, we outline the strategy of the proof of Theorem~\ref{Th-2-unconditional} and for better readability, we organize it into four steps.

\bigskip
$\bullet$ Step 1: Rescaling system~\eqref{eq:system_Eulerian_M-M1-K} to the case $K=1$.

First, we reduce the problem to the case that the constant interaction kernel is $1$ by rescaling the coordinates as follows: let $\lambda>0$ a parameter to be determined and consider the map
$$
(x,t) \mapsto \Phi_\lambda (x,t) = (\lambda x,\lambda t) ~\dot =~ (x',t')\,,
$$
with inverse:
$$
(\Phi_\lambda)^{-1} (x',t') = (x'/\lambda,t'/\lambda) = (x,t).
$$
Now, define the scaled density $\tilde \rho$ as function of $(x',t')$:
$$
(x',t') \mapsto \tilde\rho(x',t')~\dot =~ \rho\left( (\Phi_\lambda)^{-1} (x',t') \right) = \rho(x,t) \;,
$$
or equivalently,
$$
(x',t') \mapsto \tilde\rho(x',t')~\dot =~ \rho(x'/\lambda,t'/\lambda)
$$
and similarly $\tilde \vv$ and $\tilde \mm$.
Therefore, in the new coordinate system, we have the equations
\begin{equation}\label{eq:system_Eulerian_M-M1}
\begin{cases}
\partial_{t'}\tilde\rho +  \partial_{x'} (\tilde\rho \tilde\vv)  = 0, &\\
\partial_{t'} (\tilde\rho \tilde\vv)  +  \partial_{x'} \left(\tilde\rho \tilde\vv^2 + p(\tilde\rho) \right) =- \frac{K}{\lambda^2}{ \widetilde \MM \tilde\rho \left(\tilde\vv -\bar{\tilde \vv} \right)} \,
&
\end{cases} 
\end{equation}
where $\widetilde\MM$ and $\widetilde\MM_1$ are the total mass and total extended momentum for the scaled variables $(\tilde \rho(x',t'),\tilde \mm(x',t'))$. 
It is straightforward to check that $\widetilde \MM= \lambda \MM$ and $\widetilde \MM_1= \lambda  \MM_1$ while $\bar{\tilde \vv}=\bar \vv$ and the interval of the support is
$[\tilde a(t), \tilde b(t)]=[\lambda\, a(t),\lambda \, b(t)] $.
Selecting $\lambda^2=K$, we reduce to~\eqref{eq:system_Eulerian_M-M1} with kernel being $1$. To simplify the notation, from here and on, we regard that $(\rho,\mm)$ satisfies ~\eqref{eq:system_Eulerian_M-M1} for $K=1$ and without loss of generality we use $x$ and $t$, $\rho$ and $\mm$ for the rescaled variables.

Furthermore, we reduce the problem to zero average velocity. More precisely, by the definition of the average velocity $\bar \vv$ at \eqref{def:vbar}, we perform the change of variables 
\begin{equation*}
x\mapsto x-\bar \vv \,t\,,\qquad \vv\mapsto \vv- \bar \vv
\end{equation*}
that allows us to reduce to the case of $\MM_1=0$, with $\MM_1$ being defined at \eqref{cons-of-momentum}.
Indeed, in the new variables called again $(x,t)$, the total extended momentum becomes zero: 
\begin{equation}\label{eq:zero-average-momentum}
\int_{I(t)} \mm(x,t) \,dx+P_b(t)-P_a(t) = 0\qquad \forall \,t\ge0
\end{equation}
and system \eqref{eq:system_Eulerian} or~\eqref{eq:system_Eulerian_M-M1} takes the form
\begin{equation}\label{eq:system_Eulerian_M-M1-K=1}
\begin{cases}
\partial_t\rho +  \partial_x \mm  = 0, &\\[2mm]
\partial_t \mm  +  \partial_x \left(\mm^2/\rho + p(\rho) \right) = - \, \MM \rho \vv \,.
\end{cases} 
\end{equation}
Therefore, from here and on, we can assume that $\MM_1=\bar \vv =0$ and consider
system~\eqref{eq:system_Eulerian_M-M1-K=1}. We will return to the original notation in \S~\ref{Subsec:5.1}.

 \bigskip
$\bullet$ Step 2: Existence of weak solutions. 

From Theorem 1.1 in~\cite{AC_2021}, there exists an entropy weak solution $(\rho,\mm)$ with concentration along $a(t)$ and $b(t)$ to~\eqref{eq:system_Eulerian_M-M1-K=1}. According to Definition~\ref{entropy-sol} and with $\MM_1=0$, the integral identity~\eqref{S1:m-eq-phi2-pressure} of the entropy weak solution with concentration rewrites as:
\begin{align*}
\iint_{\Omega}&\left\{ \mm\phi_t
    + \left[\frac{\mm^2}{\rho}+p(\rho)  \right]\phi_x
    -\MM \mm \phi \right\}\,dx dt \nonumber\\
     &   - \int_0^\infty   \left[p(\rho(b(t)-,t)) \phi(b(t),t) -p(\rho(a(t)+,t)) \phi(a(t),t) \right] \,dt=0\;,
\end{align*}
for all test functions $\phi\in C^\infty_0(\R\times (0,\infty))$.

Following \cite[Subsect.~2.2]{AC_2021}, we transform the problem from Eulerian into Lagrangian coordinates $(y,t)$ as follows: consider the map
\begin{equation}\label{Step2 chi}
    x\to y\doteq\chi(x,t)=\int_{-\infty}^x\rho(z,t)\,dz\in[0,\MM]
\end{equation}
that is non-decreasing and it maps $\mathbb{R}\to[0,\MM]$. For the structure of solution satisfying the lower bound~\eqref{eq:density-unif-positive} that we aim for, the pseudoinverse $y\mapsto\chi^{-1}(y,t)$ is injective from $(0,\MM)$ to $(a(t),b(t))$. Hence, setting
$$
u(y,t)\doteq1/\rho(x,t),\qquad v(y,t)\doteq\vv(x,t)
$$
we arrive at the system with frictional damping
\begin{equation}\label{eq:system_Lagrangian}
\begin{cases}
\partial_\t u -  \partial_y  v  = 0, &\\
\partial_\t v  +  \partial_y (\alpha^2/u) =   - \MM v\,. & 
\\ 
\end{cases} 
\end{equation}
Following  \cite[Subsect.~3.1]{AC_2021}, we consider a sequence of approximate solutions $(u^\nu, v^\nu)$, $\nu\in\mathbb{N}$, to the initial-boundary value problem~\eqref{eq:system_Lagrangian} on $(0,\MM)\times(0,+\infty)$, together with initial data $(u_0,v_0)=(\rho_0^{-1},\vv_0)$ such that:
\begin{equation}\label{eq:init-data-lagr}
(u_0,v_0)\in BV(0,\MM)\,,\qquad \essinf_{(0,\MM)} u_0 >0\,, \qquad  \int_0^\MM v_0(y)\, dy=0\,,
\end{equation}
together with  \emph{non-reflecting} boundary conditions at the boundaries $y=0$, $y=\MM$, that read as follows:
\begin{equation}\label{eq:bc-lagrangian}
    \begin{aligned}
    y=0: & \mbox{ \emph{any} state $(u,v)(0+,t)\in (0,+\infty)\times \R$ is admissible} \\
    y=\MM: &  \mbox{ \emph{any} state $(u,v)(\MM-,t)\in (0,+\infty)\times \R$ is admissible\,. }
    \end{aligned}
\end{equation}
More precisely, let $\DT=\DT_\nu\to 0+$ and $\eta=\eta_\nu\to 0+$ two decreasing sequences.
For any $\nu\in \N$, we say that a continuous map $(u^\nu, v^\nu): [0,\infty)\to L^1((0,\MM);\R)$ is a $\nu$-approximate front tracking solution if it satisfies the following properties.

\begin{enumerate}
     \item The map $(u^\nu, v^\nu):(0,\MM)\times (0,+\infty)$ is piecewise constant, with discontinuities occurring along finitely many straight lines.
     Three types of jump discontinuities are possible: shocks, rarefactions, time-step discontinuities.
     
   \vskip 3mm
    
    \item Along a shock, let $U_\ell = (u_\ell,v_\ell)$ and $U_r = (u_r,v_r)$ be the left and right states respectively. Then either
\begin{equation*}
       u_r<u_\ell, \qquad v_r = v_\ell - \alpha \left(  \sqrt{\frac{u_\ell}{u_r}}- \sqrt{\frac{u_r}{u_\ell}} \right)
\end{equation*}
or 
\begin{equation*}
    u_r>u_\ell, \qquad v_r= v_\ell -  \alpha \left(  \sqrt{\frac{u_r}{u_\ell}} - \sqrt{\frac{u_\ell}{u_r}} 
		 \right) 
\end{equation*}
in case of a shock of the first or second family, respectively. 
Let
$$
\Lambda = -\frac{v_r-v_\ell}{u_r-u_\ell}\,,\qquad |\Lambda| = \frac{\alpha}{\sqrt{u_r u_\ell}}
$$
be the speed provided by Rankine-Hugoniot equations.

   \vskip 3mm
    
\item Along a rarefaction, the jump is non-entropic. By means of the notation above, then either
\begin{equation*}
      u_r>u_\ell, \qquad v_r = v_\ell + \alpha \ln \left(\frac {u_r} {u_\ell}\right)
\end{equation*}
or 
\begin{equation*}
    u_r<u_\ell, \qquad v_r =  v_\ell - \alpha \ln \left(\frac {u_r} {u_\ell}\right)\,.
\end{equation*}
in case of first or second family, respectively.

The Rankine-Hugoniot speed here reads as 
$$
\Lambda = -\frac{v_r-v_\ell}{u_r-u_\ell}\,,\quad |\Lambda| = \frac{\alpha}{\tilde u}\qquad \tilde u\in [\min\{u_\ell,u_r\}, \max\{u_\ell,u_r\} ]\,.
$$

\vskip 3mm

\item Let $x(t)$ be the location of a wave-front (that is, either a shock or a rarefaction).
Its propagation speed $\dot x$ satisfies $|\dot x - \Lambda|\le \eta_\nu$, where $\Lambda$ is defined above according to the cases\,. Its size is defined by
\begin{equation}\label{eq:strengths}
\eps=
\begin{cases}
\ds\frac{1}{2}\ln\left(\frac{u_r}{u_\ell}\right)& \mbox{ 1st family}\\[4mm]
\ds\frac{1}{2}\ln\left(\frac{u_\ell}{u_r}\right) &\mbox{ 2nd family}\,.
\end{cases}
\end{equation}
The size of each rarefaction satisfies $0< \eps\le \eta_\nu$\,.

\vskip 3mm

\item At each time $t^n = n\DT$, $n\ge 1$ the following holds: 
\begin{equation}\label{eq:u-v_fractional-step}
u^\nu(y,t^n+) = u^\nu(y,t^n-)\,,\qquad  v^\nu(y,t^n+) = v^\nu(y,t^n-) \left(1-\MM\DT\right)\,.
\end{equation}

\vskip 3mm

\item When a wave-front  reaches the boundary, no wave is reflected.

\vskip 3mm

\item At $t=0$, the following holds:
\begin{equation*}
\tv \left(u^\nu_0, v^\nu_0\right) \le \tv \left(u_0, v_0\right)\,,\qquad \| \left(u^\nu_0, v^\nu_0\right) -  \left(u_0, v_0\right)\|_\infty \le \frac 1 \nu\,.
\end{equation*}
\end{enumerate}

In~\cite{AC_2021}, a sequence of $\nu$-approximate front tracking solutions $(u^\nu, v^\nu)$, is constructed that has uniform BV bounds and a convergent subsequence is extracted that in the limit provides us with an entropy weak solution $(u,v)$ to~\eqref{eq:system_Lagrangian}. 

Having now this convergent sequence and its limit, we return to the Eulerian setting. For each $\nu\in\N$, we define the approximate boundaries on the $(x,t)$ plane:
\begin{equation}\label{def:a-nu}
    a^\nu(t)~\dot = ~a_0 + \int_0^t v^\nu(0+,s) ds\,.
\end{equation}
and
\begin{equation}\label{def:b-nu}
    b^\nu(t)~ \dot = ~ a^\nu(t) + \int_0^\MM u^\nu\left(y,t \right) dy \,,
    \end{equation}
the interval of support $I^\nu(t)  ~\dot = ~  \left(a^\nu(t), b^\nu(t)\right)$ and the domain
\begin{align}
\Omega^\nu   &  ~\dot = ~ \{ (x,t);\ t\ge 0\,,\  x\in I^\nu(t) \}\subset \R\times[0,+\infty)\,. \nonumber
\end{align}
Moreover, from~\cite{AC_2021}, it is shown that the map
\begin{equation}\label{eq:def-of-inverse-chi}
    x^\nu(y,t) := a^\nu(t) + \int_0^y u^\nu\left(y',t \right) dy'
    \,,\qquad y\in (0,\MM),\ t\ge 0
\end{equation}
is invertible with the inverse to be
\begin{align*}
(x,t)&\mapsto y=\chi^\nu(x,t)\\
I^\nu(t) \times[0,\infty)&\mapsto (0,\MM)  
\end{align*}
where $\chi^\nu$ is the map~\eqref{Step2 chi} corresponding to the transformation from Eulerian to Lagrangian coordinates for the approximate sequence $(u^\nu, v^\nu)$. In view of this construction, we define the approximate sequence $(\rho^\nu, \mm^\nu)$ to~\eqref{eq:system_Eulerian_M-M1-K=1} on $\mathbb{R}$ from the following relations:
\begin{align}
\label{def-rho-v-m-in}
x\in I^\nu(t):& \quad 
\begin{cases}
\rho^\nu(x,t) = \{u^\nu (\chi^\nu(x,t),t)\}^{-1}
\,,&\\[1mm]
\mm^\nu(x,t)  = \rho^\nu(x,t) v^\nu(x,t)\,,
\end{cases}
    \\[2mm]
   \label{def-rho-m-out}
     x\not\in I^\nu(t):& \qquad  \rho^\nu(x,t) =0  ~= \mm^\nu(x,t)\,. 
\end{align}
and also consider the approximate extended momentum $\hat\mm^\nu$ using the corresponding formula to~\eqref{def:m-hat}. For the details of the construction, we refer to~\cite[Sect. 4]{AC_2021}. As it is proven in~\cite{AC_2021}, $a^\nu(\cdot)$ and $b^\nu(\cdot)$ converge uniformly on compact subsets of $[0,\infty)$ to the free boundaries $a(\cdot)$ and $b(\cdot)$ as $\nu\to\infty$. Moreover,     $(\rho^\nu,\mm^\nu)$ converges as $\nu\to\infty$, up to a subsequence, to the entropy weak solution $(\rho,\mm)$ in $L^1_{loc}(\mathbb{R}\times[0,+\infty))$ to~\eqref{eq:system_Eulerian_M-M1-K=1} with concentration along $a(t)$ and $b(t)$ in the sense of Definition~\ref{entropy-sol}.

\bigskip
$\bullet$ Step 3: Proving that the total variation decays to zero as $t\to+\infty$. 

At this step, we show that the weak solution $(\rho,\mm)$ constructed admits time asymptotic flocking. To show this, it suffices to prove that its total variation decays to zero as $t\to\infty$.

Consider an approximate solution $(u^\nu,v^\nu)$ with $\nu$ fixed, as introduced in Step 2, and fix a time $t$ at which no wave interaction occurs and different from times steps $t^n$. 
Take the fronts contained in $(0,\MM)$ that exist at time $t$ in the approximate solution constructed in \cite[\S 3]{AC_2021} and neglect the standby fronts that reached the boundaries at times $\tau$, $\tau\le t$ in Lagrangian coordinates.
Let $\{y_j\}_1^N$ be the location points of those fronts in $(0,\MM)$ such that
\begin{equation}\label{eq:disc-points}
  0<y_1< y_2<\ldots< y_{N(t)}<\MM  
\end{equation}
and are the discontinuity points of $(u^\nu,v^\nu)(\cdot,t)$, for some integer $N=N(t)$ depending on $t$. Also, let $\eps_j$ be the corresponding strength of the front located at $y_j$.  Then we have the \emph{linear functional}
\begin{equation}\label{def:Lin}
L^\nu(t) = \sum_{j=1}^{N(t)} |\eps_j|\,,
\end{equation}
that does not take into account the stand by fronts at the boundaries and let us clarify that this functional is identical to the so called \emph{inner linear functional} in~\cite{AC_2021}.
As it is shown in \cite[Lemma 3.2]{AC_2021} $L^\nu(t) $ is non-increasing in time and satisfies $L^\nu(0) \le q$ for all $\nu$, with $q$ defined at \eqref{eq:def-q}. 

From Lemma 3.1 and Lemma 3.2 in \cite{AC_2021}, the following properties hold:
\begin{equation}\label{eq:TV-bound-on-v_u}
   \frac{1}{C_1} \tv \{v^\nu(\cdot,t)\} \le \frac 12 \tv \{\ln(u^\nu)(\cdot,t) \} = L^\nu(t)\,, \qquad C_1= 2\alpha \cosh(q)\,.
\end{equation}
Therefore, the exponential decay of $L^\nu(t)$, uniformly in $\nu$ implies that the total variation of the solution decays exponentially fast as well.

We consider the right-continuous representative of $L^\nu$ in the $L^1_{loc}$ class, that is, we assume that $L^\nu(t)=L^\nu(t+)$ for all $t\ge 0$. 
Since $t\to L^\nu(t)$ is monotone, we can define
\begin{equation}\label{eq Linfty}
L^\nu_\infty  := \lim_{t\to\infty}  L^\nu(t) \,.
\end{equation}
The heart of the matter is to show that $L^\nu_\infty=0$ for all $\nu$ and that this limit is reached uniformly with respect to $\nu$. This result is stated in the following lemma:

\begin{lemma}\label{lem:2.1}
Assume that $\MM\DT_\nu<\min\{1,\frac 2{\cosh q}\}$, then the total variation $L^\nu(t)$ vanishes at infinity uniformly with respect to $\nu$ i.e. for every $\eps>0$ there exists a $T>0$ such that
\begin{equation*}
    0\le  L^\nu(t) <\eps \qquad \forall\ t>T\,,\quad \forall\ \nu\in\N\,.
\end{equation*}
\end{lemma}
The proof of this lemma is given in Section~\ref{S2.1} by a careful decay analysis of the wave fronts, across interactions, time steps and at the boundaries. From this lemma, we conclude immediately that the total variation decays to zero uniformly as $t\to\infty$.

\bigskip
$\bullet$ Step 4:  Exponentially fast time asymptotic flocking.

It remains now to establish that the time-asymptotic flocking occurs exponentially fast. 
Indeed, we can prove that the convergence established in Lemma~\ref{lem:2.1} holds exponentially fast as stated here:

\begin{lemma}\label{lem:3.2}
Assume that $\MM\DT_\nu<\min\{1,\frac 2{\cosh q}\}$, then the total variation $L^\nu(t)$ decays exponentially fast as time tends to infinity uniformly in $\nu$, i.e. there exist $\nu_0\in\mathbb{N}$ and positive constants $C_1$ and $C_2$ such that
\begin{equation}\nonumber
    L^\nu(t)\le  C_1 {e}^{-C_2 t}
\end{equation} 
for all $t>0$ and $\nu>\nu_0$.
\end{lemma}
The proof of this lemma will be given in Subsection~\ref{Subsec:5.1} as part of the proof of Theorem~\ref{Th-2-unconditional}.

From here we can conclude easily that 
the solution satisfies the properties
\begin{align}\label{Th-2exp-1}
\esssup_{x\in I(t)} |\rho(x,t)- \rho_\infty|\,,\quad 
\esssup_{x_1,x_2\in I(t)} |\vv(x_1,t) - \vv(x_2,t)|
\le C_2'e^{-C_1'  t},\,\qquad \forall\, t\,
\end{align}
for some positive constants $C_1',\,C_2'$ and $\rho_\infty>0$. The proof of Theorem~\ref{Th-2-unconditional}, that will be summarized in the final Subsection~\ref{Subsec:5.1}, follows by using the scaling of the variables in Step $1$ backwards to the original variables with any kernel $K>0$.

\smallskip
It is worth mentioning that from our analysis, one can conclude an analogous result to Theorem~\ref{Th-2-unconditional} in the Lagrangian setting:

\begin{theorem}\label{Th-3-unconditional} Let $\MM>0$ and assume that the initial data $(u_0,v_0)\in BV(0,\MM)$ satisfy \eqref{eq:init-data-lagr}. Then there exists an entropy weak solution $(u(y,t),v(y,t))$ to the initial boundary problem \eqref{eq:system_Lagrangian} satisfying the boundary conditions~\eqref{eq:bc-lagrangian} globally defined in time and total variation that decays in time to zero. More precisely, the decay occurs exponentially fast, i.e. there exists $u_\infty>0$ such that
\begin{align}\label{Th-3exp}
\esssup_{y\in (0,\MM)} |u(x,t)- u_\infty|\,,\quad 
\esssup_{y_1,y_2\in (0,\MM)} |v(y_1,t) - v(y_2,t)|
\le C_2'e^{-C_1'  t},\,\qquad \forall\, t\,
\end{align}
for some positive constants $C_1',\,C_2'$.
\end{theorem}

\Section{Decay of the linear functional to zero}\label{S2.1}
The aim of this section is to prove  Lemma~\ref{lem:2.1}, that is to show that $L^\nu(t)\to 0$ as $t\to+\infty$. Having the convergent sequence $(\rho^\nu,\mm^\nu)$ and the corresponding one $(u^\nu,v^\nu)$ in the Lagrangian setting as described in Step $2$ of Section~\ref{S3}, we consider the linear functional $L^\nu(t)$ as given at~\eqref{def:Lin}. Throughout this section, we fix $\nu$ and drop the $\nu$ dependence from the approximate solution to simplify the notation.

\subsection{The local decay of the linear functional}\label{S: 4.1}
Since  $t\to L^\nu(t)$ is  non-increasing, it is immediate that $L^\nu (t)$ has a limit $L^\nu_\infty\ge 0$ as $t\to\infty$ that is given in~\eqref{eq Linfty}. 

Hence, for every $\delta>0$ there exists a time $T^{\delta,\nu}$ such that
\begin{align}\label{eq:limit-of-LinT}
  0\le L^\nu (t_1) - L^\nu (t_2) &=  \sum_{t_1<\tau\le t_2} \left[ - \Delta L^\nu (\tau)\right] \nonumber\\
  &= 
  \sum_{t_1<\tau\le t_2} \left|\Delta L^\nu (\tau)\right| \le \delta  \ \,,
\end{align}
for all $t_2>t_1 >T^{\delta,\nu}$,
and thus,
\begin{equation}\label{eq:limit-of-Lin}
  \sum_{\tau> T^{\delta,\nu}} \left|\Delta L^\nu (\tau)\right| \le \delta\,.
\end{equation}
Roughly, this means that the variation of this functional is arbitrarily small for large $t$. Let's examine in detail the consequence of the variation in \eqref{eq:limit-of-Lin} for the possible cases that arise in the approximate scheme. We first observe that the terms with $\Delta L^\nu (\tau)<0$ can be classified as follows:
\begin{align*}
      \sum_{\tau> T^{\delta,\nu}} 
      & =  \sum_{\genfrac{}{}{0pt}{}{\tau> T^{\delta,\nu}} {\rm wave\ reaches\ bd}} 
      + \sum_{\genfrac{}{}{0pt}{}{\tau> T^{\delta,\nu}} {SR\to SR} }
      + \sum_{\genfrac{}{}{0pt}{}{\tau> T^{\delta,\nu}} {SR\to SS} } \\[2mm]
     &= ~~ \mbox{\bf Case 1} ~~~+~~~ \mbox{\bf Case 2} ~~~+~~~ \mbox{\bf Case 3} \,.
\end{align*}
Indeed, in the above cases the variation of $L^\nu$ is strictly negative, while $\Delta L^\nu(t)=0$ in the remaining cases of interaction times. More precisely, in the interactions of two shocks of the same family \fbox{$SS\to RS$}, or of different families we have $\Delta L^\nu (t)=0$ and the same result holds at the time steps. Let us note that it is not possible for two rarefaction fronts to interact due to the size of their speed.

In the following proposition, we examine the Cases \textbf{1--3} for $\tau> T^{\delta,\nu}$.
\begin{proposition}\label{prop:3.1}
Let $\delta>0$ and $T^{\delta,\nu}$ be as above so that~\eqref{eq:limit-of-LinT} is satisfied. Then the followings hold true:
\begin{enumerate}
\item[{\bf Case 1:}] The total decrease of $L^\nu$ due to waves reaching the boundaries for times greater than $T^{\delta,\nu}$ is bounded by $\delta$.
\item[{\bf Case 2:}] If a shock and a rarefaction of the same family are interacting at time $\tau$, resulting into a {\bf rarefaction} of the same family, then the decrease of $L^\nu$ is
\begin{equation} \label{eq:shock-dissipation-SR-SR}
\left|\Delta L^\nu (\tau)\right| = 2|\alpha|\,.
\end{equation}
where $\alpha$ is the size of the incoming shock.
\item[{\bf Case 3:}] If a shock and rarefaction of the same family are interacting at time $\tau$, resulting into a {\bf shock} of the same family, then the decrease of $L^\nu$ is
\begin{equation}\label{eq:SR-to-SS}
\left|\Delta L^\nu (\tau)\right| \ge \frac{4 \beta}{\cosh(q)+1}
\end{equation}
where $\beta$ is the size of the incoming rarefaction and $q$ is given in~\eqref{eq:def-q}.
\end{enumerate}
\end{proposition}

\begin{proof} In {\bf Case 1}, we consider the variation of $L^\nu$ at all times $T^{\delta,\nu}$ when waves reach the boundaries. As a consequence of \eqref{eq:limit-of-Lin} we have
\begin{equation*}
    \sum_{\genfrac{}{}{0pt}{}{\tau> T^{\delta,\nu}} {\rm wave\ reaches\ bd}} \left|\Delta L^\nu (\tau)\right|  \le \delta
\end{equation*}
In other words, the quantity of waves that reaches the boundary becomes arbitrarily small for large $t$. Hence, if a large amount of waves is still present for large $t$, it must dissipate before reaching the boundary.

\smallskip
Next, in {\bf Case 2}, we treat the interaction of shock and rarefaction of the same family that results into a {\bf rarefaction} of the same family.

Let us first denote by 

- $\alpha$, $\beta$: the sizes of the incoming waves, shock and rarefaction, respectively;

- $\eps^+$ the size of the propagating wave of the same family, that is a rarefaction here, 
and by $\eps_{refl}$ the size of the reflected wave. 

We aim to show~\eqref{eq:shock-dissipation-SR-SR} that quantifies the amount of \textbf{cancellation} of wave strength in the linear functional due to the interaction of the shock and the rarefaction of the same family.

Indeed, the following basic identities hold, in all cases:
\begin{align}
\eps^+ -  \eps_{refl} & =   \alpha + \beta\,,
\label{eq:one}
\\
h(\eps_{refl}) + h(\eps^+) & =  h(\alpha) + h(\beta)\,. \label{eq:two}
\end{align}
In this \textbf{Case 2}, according to our notation, we have $\alpha<0<\beta$ and  
$\eps^+>0$,  $\eps_{refl}<0$ and hence from \eqref{eq:one}, we get
$$
|\eps^+| + |\eps_{refl}|  =  - |\alpha| + |\beta|\;.
$$
Therefore
\begin{align*}
\left|\Delta L^\nu (\tau)\right| =  |\alpha| + \underbrace{|\beta| - |\eps^+|}_{= |\alpha| + |\eps_{refl}| } - |\eps_{refl}| 
&= 2|\alpha|
\end{align*}
which is \eqref{eq:shock-dissipation-SR-SR}.
Observe that, in the case above, each incoming shock is smaller than the incoming rarefaction, and then it is smaller that $\eta_\nu$. However, the estimate above implies that, if there is an interaction of this type, the size of the incoming shock becomes arbitrarily small for large $t$.

\smallskip
Last, in {\bf Case 3}, we examine the interaction of shock and rarefaction of the same family, resulting into a {\bf shock} of the same family.

With the notation of {\bf Case 2}, here one has that $\eps^+<0$,  while $\eps_{refl}<0$ as in {\bf Case 2}. 
Now, from \eqref{eq:one}, we get
$$
 |\eps_{refl}| - |\eps^+| =  - |\alpha| + |\beta|\,,
$$
that is
$$
 |\alpha|  - |\eps^+| =    |\beta| -  |\eps_{refl}|\,.
$$
Therefore, it follows that
\begin{align}\nonumber
\left|\Delta L^\nu (\tau)\right| & =  |\alpha| - |\eps^+| + |\beta|  - |\eps_{refl}| \\
& = 2 \left( |\alpha| - |\eps^+|\right) 
= 2 \left(  |\beta|  - |\eps_{refl}|  \right) \,.\label{eq:CASE3_SR-SS}
\end{align}
Following \cite[(5.5)]{ABCD_JEE_2015}, see also Rem.~2.6 in \cite{AC_2021}, the sizes of the reflected and incoming waves satisfy
\begin{equation}\label{eq:chi_def}
    |\eps_{refl}|\le c(|\alpha|) \min\left\{ |\alpha|, |\beta| \right\}\,,\qquad c(x) = \frac{\cosh(x)-1}{\cosh(x)+1}\,.
\end{equation}
Therefore
\begin{align*}
\left|\Delta L^\nu (\tau)\right| & 
\ge 2 |\beta| \left(1 - c(|\alpha|  \right) = \frac{4 |\beta|}{\cosh(|\alpha|)+1} \ge \frac{4 |\beta|}{\cosh(q)+1}
\end{align*}
which is \eqref{eq:SR-to-SS}. The proof of Proposition~\ref{prop:3.1} is complete.
\end{proof}

\subsection{Rarefaction part vanishing at infinity}
We introduce the functionals
\begin{equation}\label{def.SR}
R^\nu(t) := \sum_{j=1,\ \eps_j>0}^{N(t)} |\eps_j|\,,\qquad 
S^\nu(t): = \sum_{j=1,\ \eps_j<0}^{N(t)} |\eps_j|\,,
\end{equation}
that correspond to the total amount of rarefactions and shocks, respectively, of both families present at time $t$ which is not of interaction and different from time steps. Across those times where these functionals have a jump, we choose them to be right-continuous and one can view these jumps either within the domain $I(t)$ in Eulerian variables or $(0,\MM)$ in Lagrangian. The functional $S^\nu(t)$ will be needed in our analysis in the following subsection.

\begin{lemma}\label{lem:rar-vanishes-for-large-t}
Assume that $\MM\DT_\nu<1$, then 
\begin{enumerate}
    \item[$(i)$] there exist $\bar T>0$ independent of $\nu$ such that
\begin{equation}\label{hyp:Rtto0L}
R^\nu(t)\le 
C_R \sum_{t< \tau\le t+\bar T}|\Delta L^\nu(\tau)|
\end{equation}
for all $t>0$, where 
\begin{equation}\label{hyp:RK}
C_R \, \dot =\, C e^{\MM \bar T}\,,\qquad 
C=\min\big\{1, \frac{\cosh(q)+1}{4}
\big\}\;.
\end{equation}
    \item[$(ii)$] the total size of rarefaction waves vanishes at infinity, i.e 
\begin{equation}\label{hyp:Rtto02}
   R^\nu(t)\to 0\qquad \text{as }t\to\infty\;.
\end{equation}
\end{enumerate}
\end{lemma}
\begin{proof}
Suppose that $\MM\DT_\nu<1$. The main point in the proof is to show that there exist $\bar T>0$ independent of $\nu$ such that~\eqref{hyp:Rtto0L} holds true.
Indeed, having established~\eqref{hyp:Rtto0L}, then we can conclude easily that $  R^\nu(t)$ tends to zero as $t\to\infty$. Indeed, for every $\delta>0$, we can take $T^{\delta,\nu}$ as in \eqref{eq:limit-of-Lin} to arrive immediately at
\begin{equation}\label{hyp:Rtto0}
R^\nu(t)\le C e^{\MM \bar T} \delta  
\end{equation}
for all times $t>T^{\delta,\nu}$. Thus,~\eqref{hyp:Rtto02} follows immediately.

We attack the proof of~\eqref{hyp:Rtto0L} by splitting it into two parts. First, we study the part of the total size of rarefaction waves, denoted by $R_2^\nu(t)$, that corresponds only to waves of the second family, i.e $\eps\in\mathcal{R}_2$. Our aim is to show that $R_2^\nu(t)$ satisfies~\eqref{hyp:Rtto0L} as $t\to\infty$. A similar argument treats the rarefaction waves of the first family $R_1^\nu(t)$ and then estimate~\eqref{hyp:Rtto0L} follows by writing $ R^\nu(t)= R_1^\nu(t)+ R_2^\nu(t)$.

To begin with, we recall that the approximate solution $\{(u,v)\}_\nu$ satisfies
$$
u_{inf}<u(y,t)<u_{sup}, \qquad \forall\,\, y\in(0,\MM),\,\, t>0
$$
where the constants $u_{inf}$ and $u_{sup}$ are independent of $\nu$ and $\Delta t_\nu$ as well as the characteristic speeds $\lambda_1^\nu(u,v)=-\alpha/u$ and $\lambda_2^\nu(u,v)=\alpha/u$. Hence, 
$$
\lambda_2^\nu(u,v)\ge\frac{\alpha}{u_{sup}}=:\lambda^*>0,\qquad \qquad \forall\,\, y\in(0,\MM),\,\, t>0,
$$
for all $\nu$. Now set 
\begin{equation}\label{def:bar-lambda_bar-T}
\bar\lambda:=\frac{1}{2}\lambda^*>0\,,\qquad \bar T:=\MM/\bar\lambda
\end{equation}
and consider the triangular region in the Lagrangian coordinates $(y,\tau)$ that is described by 
\begin{equation}\label{def:Gamma}
    \Gamma_2:= \left\{(y,\tau)\in (0,\MM)\times (t,t+\bar T);\ y\ge\bar\lambda\cdot (\tau-t)=:\bar y(\tau)
    \right\}\,.
\end{equation}
Following this construction, we observe that by the choice of the slope $\bar\lambda$ all $2$-waves present at time $t$ will be trapped in $\Gamma_2$ and cannot escape from it.

For every $\tau\in (t,t+\bar T)$, define the subinterval $J(\tau):=(\bar y(\tau),\MM)$ of $(0,\MM)$ and consider the total size of $2$-rarefaction waves in the interval $J(\tau)$, i.e.
\begin{equation}\label{def.R2}
\widetilde R_2^\nu(\tau) = \sum_{ \eps_j\in\mathcal{R}_2,\ J(\tau-)} \eps_j\,.
\end{equation}
Observe that $\widetilde R_2^\nu(t)=R_2^\nu(t)$ and every front that is in this triangular region $\Gamma_2$ at time $\tau$ cannot escape, in the sense that it cannot cross the line $y= \bar \lambda\cdot (\tau-t)$ even after possible interactions. This means that after possibly interacting with other fronts, it would end up reaching the boundary $y=\MM$ before the time $t+\bar T$ and ending its lifespan. This allows us to write the following expansion for $\widetilde R_2^\nu(\tau)$:
\begin{align}\label{R2tildevar}
    0< \widetilde R_2^\nu(t) &= \widetilde R_2^\nu(t+\bar{T})- \sum_{
    \genfrac{}{}{0pt}{}{0<\tau-t<\bar T} {\tau\not= t^n, y= \MM}
    } \Delta \widetilde R_2^\nu(\tau) 
        -\sum_{
    \genfrac{}{}{0pt}{}{0<\tau-t<\bar T} {\tau\not= t^n, y=\bar y(\tau)}
    } \Delta \widetilde R_2^\nu(\tau)\nonumber\\
    & - \sum_{
    \genfrac{}{}{0pt}{} {t<t^n<t+\bar T}{}
    } \Delta \widetilde R_2^\nu(t^n)
    - \sum_{
    \genfrac{}{}{0pt}{}{0<\tau-t<\bar T} {\tau\not= t^n, y\ne \MM}
    } \Delta \widetilde R_2^\nu(\tau) \;.
\end{align}
We note that the third term accounts to new $2$-waves that may enter the region $\Gamma_2$ through the lateral side $y=\bar y(\tau)$. See Figure~\ref{Fig1} for an illustration of the several cases of $2$-waves as decomposed into~\eqref{R2tildevar}.
\begin{figure}\label{Fig1}
\floatbox[{\capbeside\thisfloatsetup{capbesideposition={right,center},capbesidewidth=8cm}}]{figure}[\FBwidth]
{
\caption{In this triangular region, in Lagrangian coordinates, for $\tau\in[t,t+\bar T]$ and $y\in J(
\tau)=(\bar y(\tau),\MM)$, we illustrate the waves of the second family that are trapped according to the following cases: \newline
Cases $(I)$ and $(II)$: $2$-waves meet the boundary $y=\MM$;\newline Case $(III)$: $2$-waves enter the region through the lateral side $y=\bar y(\tau)$;\newline Case $(IV)$: wave interactions and last,\newline Case $(V)$: at time steps.
\newline
Dashed lines: 1-waves, 
\newline
Solid lines: 2-waves}
}
{\scalebox{.6}{ \input{triangle.pspdftex}
}}
\end{figure}

Next, we introduce the piecewise constant weight
\begin{equation}\label{def:W}
W(\tau):= \exp\left\{\MM  \cdot n(\tau) \DT_\nu -\MM t\right\}  \,,  
\end{equation}
for $\tau\in(t,t+\bar T)$, where $n(\tau)$ is the integer part of $\tau/\DT_\nu$ satisfying
$$
n(\tau) \DT_\nu\le\tau<n(\tau) \DT_\nu+\DT_\nu
,\,\qquad
n(t^n+)-n(t^n-)=1\;.$$

The function $W(\tau)$  is increasing in $\tau$ and discontinuous at the time steps:
\begin{equation}\label{eq:WR2-at-timestep}
    W(t^n+) - W(t^n-) = W(t^n-) \left(e^{\MM\DT_\nu} -1 \right)\,,
\end{equation}
satisfying the bounds
\begin{equation}\label{eq:bound-on-WR2}
1   
\le W(\tau) \le e^{\MM \bar T}\,.
\end{equation}
We introduce now the weighted expression of $\widetilde R_2^\nu$ in the region $\Gamma_2$, i.e.
\begin{equation}\label{def:F2}
F_2^\nu(\tau):= \widetilde R_2^\nu(\tau) W(\tau),\qquad \tau\in(t,t+\bar T)    
\end{equation}
and in view of the above analysis for $\widetilde R_2(\tau)$, this satisfies
\begin{align}
    0\le F_2^\nu(t) &= F_2^\nu(t+\bar T) - \sum_{
    \genfrac{}{}{0pt}{}{0<\tau-t<\bar T} {\tau\not= t^n, y= \MM}
    }  \left(\Delta\widetilde R_2^\nu(\tau)\right) W(\tau) \nonumber\\
    &\qquad -\sum_{
    \genfrac{}{}{0pt}{}{0<\tau-t<\bar T} {\tau\not= t^n, y=\bar y(\tau)}
    }\left( \Delta \widetilde R_2^\nu(\tau) \right) W(\tau) \label{eq:variation-of-R2-W-0_III}
    \\
      &\qquad - \sum_{
    \genfrac{}{}{0pt}{}{0<\tau-t<\bar T} {\tau\not= t^n, \bar y<y< \MM}
    }  \left(\Delta\widetilde R_2^\nu(\tau)\right) W(\tau)\label{eq:variation-of-R2-W-0_IV} 
    \\
    &\qquad - \sum_{
    \genfrac{}{}{0pt}{}{n} {t<t^n<t+\bar T}
    } \left[ \widetilde R_2^\nu(t^n+)W(t^n+) -\widetilde R_2^\nu(t^n-)W(t^n-) \right] \nonumber
    \\
    &= (I) + (II) + (III)+(IV)+(V) \label{eq:variation-of-R2-W-0}
\end{align}
for $t>0$. Now, we investigate the terms $(I)$-$(V)$:

\begin{itemize}
    \item [$(I)$+$(II)$] We observe that $\widetilde R_2^\nu(t+\bar T)$ involves  waves squeezed into $J(t+\bar T)$ that is at the point $y=\MM$. Hence, $(I)$ can be treated together with $(II)$ that accounts to the waves reaching the boundary $y=\MM$ at previous times $\tau<t+\bar T$. In both cases, we have
    $$
    \widetilde R_2^\nu(t+\bar T)=-\Delta L^\nu(t+\bar T)>0,\qquad
    \Delta \widetilde R_2^\nu(\tau)=\Delta L^\nu(\tau)<0,\quad \tau<t+\bar T
    $$
   at the boundary $y=\MM$.
 By Case $1$ in \S~\ref{S2.1}, we get
\begin{equation*}
   (I)+(II)=\sum_{
    \genfrac{}{}{0pt}{}{0<\tau-t\le \bar T} {\tau\not= t^n, y= M}} \left|\Delta L^\nu(\tau)\right| W(\tau) \le e^{M  \bar T}
    \sum_{
    \genfrac{}{}{0pt}{}{0<\tau-t\le \bar T} {\tau\not= t^n, y= M}}\left|\Delta L^\nu(\tau)\right|  
\end{equation*}
for $t>0$.
\item [$(III)$] Regarding the next term \eqref{eq:variation-of-R2-W-0_III}, we observe that $\Delta \widetilde R_2^\nu(\tau)>0$ since new $2$-waves enter the region $\Gamma_2$ along $y=\bar y(\tau)$, hence the contribution of this term is negative, i.e. $(III)\le 0$. 
            \item [$(IV)$]
The next term \eqref{eq:variation-of-R2-W-0_IV} combines the variation of $\widetilde R_2^\nu(\tau)$ at interaction times $\tau\in(t,t+\bar T)$. The interactions that may produce a change in $\widetilde R_2^\nu(\tau)$ are: 

(i) Interactions of two shock-waves of the first family that result into a reflected wave $\eps_{refl}\in\mathcal{R}_2$. 
In this situation, $\Delta\widetilde R_2^\nu(\tau)>0$.

(ii) Interactions of waves of the second family that belong to Case $2$ in~\S~\ref{S: 4.1}. In this situation, denoting by $\beta^-$ the incoming rarefaction, $\beta^+$ the outgoing rarefaction, $\alpha$ the incoming shock, and $\eps_{refl}$ as usual the reflected $1$-wave, we get
$$
0<-\Delta\widetilde R_2^\nu(\tau)=-|\beta^+|+|\beta^-|=|\alpha|+|\eps_{refl}|\le 2|\alpha|=|\Delta L^\nu(\tau)|
\;.$$

(iii) Last, interactions of waves of the second family that belong to Case $3$ in~\S~\ref{S: 4.1}. Adapting in this situation again the previous notation, we have
$$
0<-\Delta\widetilde R_2^\nu(\tau)=|\beta^-|\le \frac{\cosh(q)+1}{4} \left|\Delta L^\nu(\tau)\right|\,. $$
We note that the shock-shock interaction of the second family is not relevant in the variation of $\widetilde R_2(\tau)$. Combining the above estimates, we get
\begin{equation*}
- \sum_{
    \genfrac{}{}{0pt}{}{0<\tau-t<\bar T} {\tau\not= t^n, y< M}
    }  \left(\Delta\widetilde R_2^\nu(\tau)\right) W(\tau)\le C e^{\MM  \bar T} \sum_{
    \genfrac{}{}{0pt}{}{0<\tau-t<\bar T} {\tau\not= t^n, y< \MM}
    }  \left|\Delta L^\nu(\tau)\right|\,,
\end{equation*}    
where $C$ is given at~\eqref{hyp:RK}.
\item [$(V)$] 
We claim that 
    $$
     \Delta F^\nu_2(t^n)=\widetilde R_2(t^n+)W(t^n+) - \widetilde R_2(t^n-)W(t^n-)  \ge 0\;,$$
    for all time steps $t^n\in(t,t+\bar T)$.
    Indeed, across a time step $t^n$, $ \widetilde R_2^\nu (t^n)$ changes if a rarefaction of the second family $\eps_2^-\in\mathcal{R}_2$ reaches the time step $t^n$ and it gets updated to 
    $\eps_2^+\in\mathcal{R}_2$ or a shock of the first family $\eps_1^-\in\mathcal{S}_1$ reaches the time step $t^n$ and after the update a new rarefaction $\eps_2^+\in\mathcal{R}_2$ is produced for the second family. Taking both situations into consideration, we write
    $$
    \Delta  \widetilde R_2^\nu (t^n) =\sum_{
     {\eps_2^-\in\mathcal{R}_2}
    }|\eps^+_2|-|\eps_2^-|+\sum_{
     {\eps_1^-\in\mathcal{S}_1}
    } |\eps_2^+|\ge \sum_{
     {\eps_2^-\in\mathcal{R}_2}
    }|\eps^+_2|-|\eps_2^-|\;.
    $$
    By  \cite[Proposition~ 3.2]{AC_2021} 
    if $\eps_2^-\in\mathcal{R}_2$ reaches the time step $t^n$, we have 
    $$ 0>|\eps^+_2|-|\eps_2^-|=\eps^+_2-\eps_2^-=-|\eps_1^+|\ge-\frac{1}{2} \MM\DT_\nu|\eps_2^-|\;.
    $$
    Hence, 
        $$
    \Delta  \widetilde R_2^\nu (t^n) \ge-\frac{1}{2}\MM\DT_\nu\,\widetilde R_2^\nu (t^n-)
    $$
    Now combining the above estimate with~\eqref{eq:WR2-at-timestep}, we arrive at
\begin{align*}
    \Delta F_2^\nu(t^n)=    \Delta (  \widetilde R_2^\nu W)(t^n) & = \widetilde R_2^\nu(t^n+) \Delta W(t^n) + W(t^n-)  \Delta\widetilde R_2^\nu(t^n)\\
    &= W(t^n-) \left[\widetilde R_2^\nu (t^n+) \left(e^{\MM \DT_\nu} -1 \right) +  \Delta\widetilde R_2^\nu(t^n) \right]\\
    &= W(t^n-) \left[\widetilde R_2^\nu(t^n+) e^{\MM \DT_\nu}   -\widetilde R_2^\nu(t^n-) \right]\\
 &\ge W(t^n-) \widetilde R_2^\nu(t^n-)  \left[ (1 - \frac \MM2 \DT_\nu) e^{\MM \DT_\nu} -1 \right]\\
 &\ge W^*(t^n-)\widetilde R_2^\nu(t^n-) 
   \cdot \frac{\MM\DT_\nu}{2}\cdot
(1-\MM\DT_\nu )\ge 0\;,
\end{align*}
since $\MM\DT_\nu<1$. The claim is proven.

\end{itemize}
In view of the above analysis, we arrive at
$$
0\le\widetilde R_2^\nu(t)\le \widetilde R_2^\nu(t)
+\sum_{
    \genfrac{}{}{0pt}{}{0<\tau-t<\bar T} {\tau\not= t^n, y=\bar y(\tau)}
    } \Delta \widetilde R_2^\nu(\tau) 
    \le F^\nu_2(t)+\sum_{
    \genfrac{}{}{0pt}{}{0<\tau-t<\bar T} {\tau\not= t^n, y=\bar y(\tau)}
    } \Delta \widetilde R_2^\nu(\tau) ,\qquad \forall t>0
   $$
and
$$
F^\nu_2(t)+\sum_{
    \genfrac{}{}{0pt}{}{0<\tau-t<\bar T} {\tau\not= t^n, y=\bar y(\tau)}
    } \Delta \widetilde R_2^\nu(\tau) \le C e^{\MM \bar T}  \sum_{
    \genfrac{}{}{0pt}{}{t< \tau\le t+\bar T} {}
    }  \left|\Delta L^\nu(\tau)\right| ,\qquad \forall \,t>0\,.
$$
Recalling that $R_2^\nu(t)=\widetilde{R}_2^\nu(t)$, we conclude
   \begin{align}\label{eq:variation-of-R2-W-0final}
0\le R_2^\nu(t) \le C e^{\MM \bar T} 
\sum_{
    \genfrac{}{}{0pt}{}{t< \tau\le t+\bar T } {}
    }  \left|\Delta L^\nu(\tau)\right| ,\qquad \forall t>0\;.
\end{align}
Here, we observe that the summation runs only over the cases (I), (II) and (IV) as discussed above. Hence this involves only $2$-waves reaching the boundary $y=\MM$ and/or interactions of $2$-waves.

Similarly, one can establish  estimate~\eqref{eq:variation-of-R2-W-0final} for $R_1^\nu(t)$ working in the region
\begin{equation}\label{def:Gamma1}
    \Gamma_1:= \left\{(y,\tau)\in (0,\MM)\times (t,t+\bar T);\ y\le \MM-\bar\lambda\cdot (\tau-t)=:\bar y_1(\tau)
    \right\}\,,
\end{equation}
since $\lambda_1^\nu(u,v)\le\lambda^*$, for all $t>0$ and $\nu$. In $\Gamma_1$, estimate~\eqref{eq:variation-of-R2-W-0final} for $R_1^\nu(t)$ holds
having the summation running over the times $\tau$ that a $1$-wave meets the boundary $y=0$ or for interactions of waves of the first family. Thus $ R^\nu(t)= R_1^\nu(t)+ R_2^\nu(t)$ satisfies~\eqref{hyp:Rtto0L}. 
We note here that $C_R \, \dot =\, C e^{\MM \bar T}$ because the summation in~\eqref{eq:variation-of-R2-W-0final} runs over the cases of the second family only but now in~\eqref{hyp:Rtto0L} it combines the fronts for both families.
The proof of Lemma~\ref{lem:rar-vanishes-for-large-t} is complete.
\end{proof}

\subsection{Uniform long-time behavior}
In this subsection, we address the proof of Lemma~\ref{lem:2.1}, i.e. we prove that $ L^\nu(t) \to L^\nu_\infty=0$ as $t\to\infty$ uniformly in $\nu$.

Since the sequence $\{L^\nu(t)\}_{\nu\in\N}$ is uniformly bounded and non-increasing, there exists a function $L^\infty(t)$ that is, up to a subsequence, the pointwise limit of $L^\nu(t)$,
\begin{equation}\label{def:L-in_infty}
    L^\nu(t)\xrightarrow{\nu\to\infty} L^\infty(t) \qquad \forall\ t\ge 0\,.
\end{equation}
As it is shown in the next proposition, it suffices to establish that $L^\infty(t)$ decays to zero as $t\to\infty$ since this immediately implies $L^\nu_\infty=0$.

\begin{proposition}\label{prop:unif-conv-of-L}
If $L^\infty(t)\to0$ as $t\to\infty$, then $L^\nu(t)\to0$ as $t\to\infty$ uniformly in $\nu$.
\end{proposition}
\begin{proof}
Let $\eps>0$. Then there exists $T^\infty_\eps>0$ such that 
$$
0\le L^\infty(t)<\eps,\qquad\forall t>T^\infty_\eps
$$
By the pointwise convergence of $L^\nu(t)$, there exists $\bar\nu=\bar\nu(T^\infty_\eps)$ so that
$$
|L^\nu(T^\infty_\eps)-L^\infty(T^\infty_\eps)|<\eps
$$
for all $\nu\ge\bar \nu$. However, since  $L^\nu(t)$ is non-increasing in time, we have
$$
L^\nu(t)\le L^\nu(T^\infty_\eps)\le L^\infty(T^\infty_\eps)+\eps<2\eps
$$
for all $t\ge T^\infty_\eps$ and $\nu\ge\bar\nu$. This completes the proof.
\end{proof}

In the following lemma, we show that for the weak solution $(\rho,\mm)$ already constructed via the front tracking algorithm in the sense of Definition~\ref{entropy-sol}, the functional $L^\infty(t)$ decays to zero as $t\to\infty$ and this is accomplished by studying the wave decay after a time threshold related to $\bar T$ and within a time length that is large enough. In this setting, we manage to show that the rarefaction part, that is vanishing, will force the total variation to vanish as well.

\begin{lemma}\label{L4.2}
Let $(\rho,\mm)$ be an entropy weak solution with concentration along $a(t)$ and $b(t)$ to~\eqref{eq:system_Eulerian_M-M1-K=1} 
obtained as a limit of wave-front tracking approximation
with the associated functional 
 $  L^\infty(t)$ as given by~\eqref{def:L-in_infty}. Then the following limit holds true
\begin{equation}\label{eq:L-infty-to0}
    L^\infty(t)\to 0 \qquad t\to+\infty\,.
\end{equation}
\end{lemma}

\begin{proof} To prove \eqref{eq:L-infty-to0} we proceed by contradiction and assume that 
\begin{equation}\label{eq:L-infty-tonot0}
\lim_{t\to\infty}  L^\infty(t) =: \bar L >0 \,.
\end{equation}
Let $\eta>0$. By the existence of the limit~\eqref{eq:L-infty-tonot0}, there exists $T_{\eta}^\infty>0$ such that 
\begin{equation}\label{eq:limit-of-Lin-infty}
0\le\, L^\infty(t) - \bar L = \sum_{\tau> t} \left|\Delta L^\infty(\tau)\right| \le \frac\eta 2\qquad \forall\, t> T_\eta^\infty\,.
\end{equation}
Now, fix a time $t_1$  with $t_1\ge T_\eta^\infty$ and recall the quantity $\bar T$, independent of $\nu$, that has been introduced in \eqref{def:bar-lambda_bar-T} during the proof of Lemma~\ref{lem:rar-vanishes-for-large-t} and represents the time length for which \eqref{hyp:Rtto0L} holds. 
From the monotonicity of $L^\infty(t)$ and~\eqref{eq:limit-of-Lin-infty}, it follows
\begin{equation*}
    L^\infty(t_1) - L^\infty(t_1+ 2 \bar T) 
    \le  L^\infty(T_{\eta}^\infty) -\bar L\le \frac\eta 2\,.
\end{equation*}
By the pointwise convergence of the sequence $\{L^\nu\}_{\nu\in\N}$, there exists $\bar\nu=\bar\nu(t_1,\bar T,\eta)$ large enough such that
\begin{equation}\label{eq:approx_t_t+2T}
|L^\nu(t_1)-L^\infty(t_1)|<\frac \eta 4,\qquad |L^\nu(t_1+ 2 \bar T)-L^\infty(t_1+ 2 \bar T)|<\frac \eta 4
\end{equation}
and hence
\begin{equation}\label{eq:L-nu-t1barT}
    0\le L^\nu(t_1) - L^\nu(t_1+ 2 \bar T) 
    =\sum_{t_1 < 
    \tau \le 
    t_1+2\bar T}|\Delta L^\nu(\tau)|
    \le \eta \,,
\end{equation}
for all $\nu\ge\bar\nu$\,, where we used the right continuity of $L^\nu(\cdot)$.    
From \eqref{eq:L-nu-t1barT} we get
\begin{equation}\label{eq:L-nu-t1barT-bis}
 0\le L^\nu(t) - L^\nu(t+  \bar T) \le L^\nu(t_1) - L^\nu(t_1+ 2 \bar T) \le \eta
\end{equation}
for all $t\in [t_1,t_1+\bar T]$, since $[t,t+\bar T]\subset [t_1,t_1+2\bar T] $.

Let now $\bar \nu_1\ge\bar\nu$ such that $M\DT_\nu<1$ for all $\nu\ge\bar\nu_1$. From estimate~\eqref{hyp:Rtto0L} in Lemma~\ref{lem:rar-vanishes-for-large-t} and~\eqref{eq:L-nu-t1barT-bis}, 
we immediately deduce
\begin{equation}\label{eq:R-nu-t1barT}
R^\nu(t)\le C_R 
\sum_{t< \tau\le t+\bar T}|\Delta L^\nu(\tau)| \le C_R  \eta
\end{equation}
for all $t\in [t_1,t_1+\bar T]$ and $\nu\ge\bar\nu_1$, with $C_R$ independent of $\nu$ given at~\eqref{hyp:RK}. On the other hand, using the second inequality in \eqref{eq:approx_t_t+2T} for all $\nu\ge\bar\nu_1$ and the monotonicity of $L^\nu$ in time and~\eqref{eq:limit-of-Lin-infty}, we get
\begin{align*}
    L^\nu(t)&\ge L^\nu(t_1+2\bar T)\\
    &= \left(  L^\nu(t_1+2\bar T) - L^\infty(t_1+2\bar T)\right) +\left( L^\infty(t_1+2\bar T)- \bar L\right)+ \bar L
    \\
    &\ge 
    -\frac \eta 4 + 0+\bar L
\end{align*}
for all $ t\in [t_1, t_1+\bar T]$ and $\nu\ge\bar\nu_1$.

Next, we recall the shock part $S^\nu(t)$ defined in~\eqref{def.SR}. Having the above estimates, we can deduce that for $\nu$ large enough, $S^\nu(t)$ is uniformly positive on the time interval $(t_1,t_1+\bar T)$ while the rarefaction part $R^\nu(t)$ is small. More precisely, for $\eta>0$ small enough such that
$$
 \eta < \frac{\bar L}{ 2 \bar C}\qquad \text{with }\bar C \,\dot =\, C_R + \frac 14
$$
one has
$$
S^\nu(t)=L^\nu(t)-R^\nu(t)\ge \bar L-\bar C \eta  >\frac{\bar L}{2}>0
$$
for $t\in[t_1,t_1+\bar T]$ and $\nu\ge\bar\nu_1$. This immediately implies that we have the following lower bound 
\begin{equation}\label{eq:Snutimestepsadd}
    \sum_{t_1< t^n\le t_1+\bar T} \MM\DT S^\nu(t^n-)\,>\,\frac{\bar L}{2} \MM \bar T
\end{equation}
true for all $\nu\ge\bar\nu_1$, noting that the lower bound is independent of $\nu$.

 In what follows, we show that \eqref{eq:Snutimestepsadd} leads to a contradiction. Roughly speaking, a uniformly positive variation of shocks $S^\nu(t)$ in the interval $(t_1,t_1+\bar T)$ generates a uniformly positive amount of rarefactions at the time steps, due to the source term. However, such a uniform amount of $R^\nu(t)$ is in contradiction with \eqref{eq:R-nu-t1barT}.

Indeed, lets investigate the variation of $R^\nu(t)$ in the interval $t\in(t_1,t_1+\bar T)$.
\begin{equation}\label{eq:R-nu-t1barT2}
    R^\nu(t)=R^\nu(t_1)+\sum_{t_1<\tau<t} [\Delta R^\nu(\tau)]_+-[\Delta R^\nu(\tau)]_-
\end{equation}
where $[f]_+$ and $[f]_-$ denote the positive and negative part of $f$, respectively. Before we proceed, let us list the cases involved in the negative variation of $R^\nu$  as analysed in the proof of Lemma~\ref{lem:rar-vanishes-for-large-t}:\\
There are three cases for which $[\Delta R^\nu(\tau)]_-$ is nonzero and these are:

$(i)$ when there is a shock-rarefaction interaction of the same family that results to a shock or a rarefaction of the same family. In view of the analysis of the term (IV) in Lemma~\ref{lem:rar-vanishes-for-large-t}, it holds $[\Delta R^\nu(\tau)]_-\le C\,|\Delta L^\nu(\tau)|$ in this case, where $C\ge 1$ is given at~\eqref{hyp:RK}; 

$(ii)$ when a rarefaction front reaches the boundary $y=0$ or $y=\MM$. In this case, we have again $[\Delta R^\nu(\tau)]_-\le \,|\Delta L^\nu(\tau)|$ as mentioned in term $(II)$;

$(iii)$ when a rarefaction front $\eps^-$ meets a time step at $\tau=t^n$. In this case, we have $[\Delta R^\nu(\tau)]_-\le \frac 1 2 \MM \DT |\eps^-|$. One can verify this from the analysis of the term (V) in Lemma~\ref{lem:rar-vanishes-for-large-t}. Summing over all time steps and using estimate~\eqref{eq:R-nu-t1barT}, we reach
$$
\sum_{
\genfrac{}{}{0pt}{}{t_1<t^n<t} {\eps^-\in\mathcal{R}}
} [\Delta R^\nu(t^n)]_-\le
\frac \MM 2  \DT \sum_{
\genfrac{}{}{0pt}{}{t_1<t^n<t} {\eps^-\in\mathcal{R}}
}
|R^\nu(t^n-)|= \frac \MM 2  (t-t_1) C_R
\eta<\frac \MM 2  \bar T C_R 
\, \eta
$$
for all $t\in(t_1,t_1+\bar T)$ and $\nu\ge\bar\nu_1$. Combining now these three cases with~\eqref{eq:L-nu-t1barT}, we obtain 
\begin{align*}
    \sum_{t_1<\tau<t} [\Delta R^\nu(\tau)]_-  & < \frac \MM 2  \bar T C_R \, \eta + C \sum_{t_1<\tau<t} |\Delta L^\nu(\tau)|\\
& <  
\left(\frac \MM 2  \bar T C_R +C\right)\eta\;,
\end{align*}
for all $t\in(t_1,t_1+\bar T)$ and $\nu\ge\bar\nu_1$. Substituting this and estimate~\eqref{eq:R-nu-t1barT} into~\eqref{eq:R-nu-t1barT2}, we infer that the sum of the positive part is also small, i.e.
\begin{align}\nonumber
\sum_{t_1<\tau<t} [\Delta R^\nu(\tau)]_+ 
&=  R^\nu(t)- R^\nu(t_1)+\sum_{t_1<\tau<t}[\Delta R^\nu(\tau)]_-\\
&\le  R^\nu(t)+\sum_{t_1<\tau<t}[\Delta R^\nu(\tau)]_- \nonumber\\
&\le \left(C_R+\frac \MM 2  \bar T C_R +C\right)\eta\;, \label{eq:DRpos}   
\end{align}
for all $t\in(t_1,t_1+\bar T)$ and $\nu\ge\bar\nu_1$.
Let us examine now the cases involved in the positive variation of $R^\nu$, which are two:

$(i)$ when there is a shock-shock interaction of the same family and hence, a reflected rarefaction arises after the interaction. In this case $[\Delta R^\nu(\tau)]_+\ge 0$;

$(ii)$ when a shock front $\eps^-$ meets a time step at time $\tau=t^n$  and a rarefaction front $\eps_{refl}$ of the other family arises after the update. From \cite[Proposition 3.2]{AC_2021}, it holds $[\Delta R^\nu(t^n)]_+=|\eps_{refl}|\ge \MM\DT c_1(q) |\eps^-|$,
where $c_1(q)= (1+\cosh(q))^{-1}$.
\begin{figure}\label{Fig2}
\floatbox[{\capbeside\thisfloatsetup{capbesideposition={right,center},capbesidewidth=8.2cm}}]{figure}[\FBwidth]
{
\caption{This figure illustrates that the presence of shocks as time evolves gives rise to a rarefaction part that is comparable in size and not vanishing. The two cases for which the positive variation of $R^\nu(\tau)$ is estimated from below in~\eqref{eq: for fig2} are shown in this figure for $\tau\in(t,t+\bar T)$. More precisely,
\newline 
(i) after a shock-shock interaction of the same family, $[\Delta R^\nu(\tau)]_+=\eps_{refl}\ge 0$ and\newline (ii) after a shock of size $\eps^-$ gets updated at a time step, $[\Delta R^\nu(\tau)]_+\ge \MM \Delta t c_1(q) |\eps^-|$\;.
\newline
Solid lines: Shock fronts (S)
\newline
Dashed lines: Rarefaction fronts (R)
}}
{\scalebox{.6}{ \input{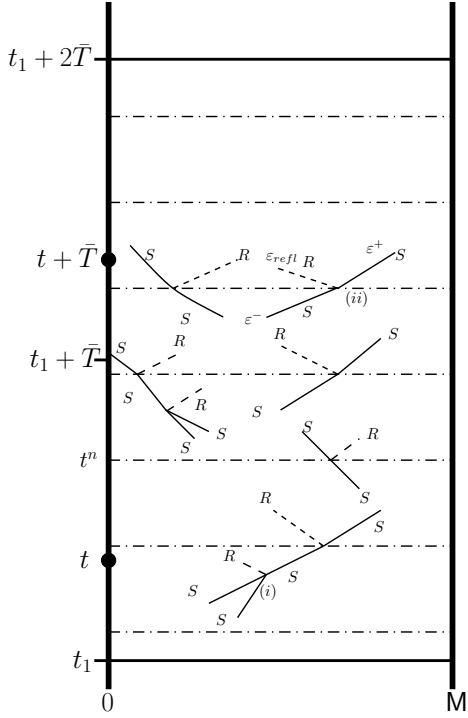}
}}
\end{figure}
Taking into account only case (ii) that involves the time steps, we get
\begin{equation}\label{eq: for fig2}
\sum_{t_1<\tau<t} [\Delta R^\nu(\tau)]_+\ge \sum_{
\genfrac{}{}{0pt}{}{t_1<t^n<t} {\eps^-\in\mathcal{S}}
} [\Delta R^\nu(t^n)]_+\ge  c_1(q) \sum_{t_1<t^n<t}  \MM\DT S^\nu(t^n-)\;,
\end{equation}
where the sum in the middle term above contains the newly produced rarefactions that arise as reflected waves of previously existing shocks.
However, using the lower bound~\eqref{eq:Snutimestepsadd}, we arrive at
$$
\sum_{t_1<\tau<t} [\Delta R^\nu(\tau)]_+\ge c_1(q)\frac{\bar L}{2} \MM \bar T>0\;,
$$
for all $t\in(t_1,t_1+\bar T)$, which contradicts~\eqref{eq:DRpos} since $\eta$ can be arbitrarily small. Thus, we deduce that $\bar L=0$ and the proof is complete.
\end{proof}

Now, Lemma~\ref{lem:2.1} follows immediately by combining Lemma~\ref{L4.2} and Proposition~\ref{prop:unif-conv-of-L}.

\Section{Exponential convergence}\label{S5}
It remains to show that the time asymptotic flocking is exponentially fast. The strategy of the proof of this decay is similar in spirit to the analysis of the previous section and relies again on the wave decay.

\begin{lemma}\label{S5Prop 5.1}
There exists $T^*>0$, such that 
\begin{equation}\label{T*star}
    L^\infty(\tau+T^*)\le \frac 12 L^\infty(\tau)\,,\qquad \forall\, \tau \ge 0 \,.
\end{equation} 
\end{lemma}

\begin{proof}
We recall the functional $L^\infty(\cdot)$ from definition \eqref{def:L-in_infty} that is a non-increasing function of time and it decays to $\bar L=0$ from Lemma~\ref{L4.2}.

For every $\tau>0$, now, we first define the time $T^{**}$ to be
\begin{equation}\label{T**star}
  T^{**}(\tau):=\inf\left\{
  t:\, t\ge \tau,\,\, L^\infty(t)\le \frac12 L^\infty(\tau)
  \right\}\;.
\end{equation} 
Here, we note that the infimum exists since $L^\infty(\tau)\ge 0$ and $L^\infty(\tau)\to0$ as $\tau\to+\infty$. Hence the set is non empty and therefore $T^{**}$ is well defined. Moreover, we state some properties of $T^{**}$: for every $\tau> 0$ with $L^\infty(\tau)>0$,
\begin{equation}\label{eq 5 T** t ineq}
    L^\infty(T^{**}(\tau)-)>\frac 12 L^\infty(\tau),\qquad L^\infty(T^{**}(\tau)+)\le\frac 12 L^\infty(\tau)\,.
\end{equation}
On the other hand, we observe that if $L^\infty(\tau)=0$ for some $\tau$, then $T^{**}(\tau)=\tau$.

Next we set
\begin{equation}\label{T*stardef}
  T^{*}:=\sup\left\{
  T^{**}(\tau)-\tau \,:\, \tau\ge 0
  \right\} \,+\, 1\;,
\end{equation} 
and we aim to prove that $T^{*}$ is finite.
\footnote{We may interpret $T^*$ in \eqref{T*stardef} as an estimate of the {\it half-life} time for $L^\infty$, in view of \eqref{T*star}. 
The $+1$ in \eqref{T*stardef} could be replaced by any constant $\sigma>0$. The limit as $\sigma\to 0+$ leads to a better estimate of such half-life time.}
Indeed, having proved that $T^{*}<+\infty$, by the definition~\eqref{T*stardef}, we find that 
$$
T^{*} + \tau \ge  T^{**}(\tau)+ 1 \qquad \forall \, \tau\,,
$$
that yields, by the monotonicity of $L^\infty$:
$$
L^\infty(T^*+\tau)\le L^\infty(T^{**}(\tau)+ 1)\le L^\infty(T^{**}(\tau)+)\le \frac12 L^\infty(\tau)
$$
for every $\tau\ge 0$. Thus, estimate~\eqref{T*star} holds true for $T^*$ given at~\eqref{T*stardef} and the proof is complete.

In view of the above, it suffices to show that $T^{*}<+\infty$. We proceed by contradiction and we assume that 
$T^{*}=+\infty$, that is, $\sup\left\{
  T^{**}(\tau)-\tau \,:\, \tau\ge 0
  \right\} = +\infty$. 
Hence, there exists a sequence $\{\tau_m\}\subset[0,\infty)$ of times for which $T^{**}(\tau_m) - \tau_m \to+\infty$ as $m\to\infty$. 
Therefore, there exists $m_1$ such that $T^{**}(\tau_m) - \tau_m>\bar T_1:=\bar T +1$   
for all $m>m_1$  and in particular, it holds $L^\infty(\tau_m)>0$ for all $m>m_1$. 
Here, we remind the reader that $\bar T$ was introduced in~Lemma~\ref{lem:rar-vanishes-for-large-t} in \eqref{def:bar-lambda_bar-T}.    

Let $\delta>0$. We claim that there exists $m_0>m_1$, and  $t_0>\tau_{m_0}$ so that 
\begin{equation}\label{claimTbarintervalnew}
    L^\infty(t_0) - L^\infty(t_0+\bar T) \le \frac \delta2\,L^\infty(\tau_{m_0})\,,\qquad t_0+\bar T < T^{**}(\tau_{m_0})
\end{equation}
and such that $L^\infty$ is continuous at $t_0$ and at $t_0+\bar T$.
Loosely speaking, estimate~\eqref{claimTbarintervalnew} shows that there exists an interval of length $\bar T$, with
$$
[t_0,t_0+\bar T]\subset \left(\tau_{m_0}, T^{**}(\tau_{m_0})\right)
$$
where $L^\infty$ is approximately constant, relatively to $L^\infty(\tau_{m_0})$. Indeed, notice that property~\eqref{eq 5 T** t ineq}$_1$ implies
\begin{equation}\label{eq:ge-ge-g}
     L^\infty(t_0)\ge  L^\infty (t_0 + \bar T)
\ge  L^\infty \left(T^{**}(\tau_{m_0})-\right) 
     >\frac 12 L^\infty(\tau_{m_0})
\end{equation}
and hence $L^\infty(t_0)/ L^\infty(\tau_{m_0})$ is uniformly positive.
After \eqref{claimTbarintervalnew} is proved, the contradiction will be reached by proving that $L^\infty$, or equivalently $L^\nu$ for large $\nu$, cannot stay approximately  constant for an interval of length $\bar T$, due to wave decay.

To prove the claim \eqref{claimTbarintervalnew}, we first notice that \eqref{eq 5 T** t ineq}$_1$ immediately implies
$$
L^\infty(\tau_m)-L^\infty(T^{**}(\tau_m)-)<\frac 12 L^\infty(\tau_m)\qquad \forall\,\, m>m_1\;.
$$
On the other hand, let $k=k(m)$ be the integer part of $(T^{**}(\tau_m)-\tau_m)/\bar T_1$ and notice that $k(m)\ge 1$ for $m>m_1$. Then one has
\begin{align*}
L^\infty(\tau_m)-L^\infty(T^{**}(\tau_m)-)
&\ge 
\sum_{i=1}^{k}
\left[
L^{\infty}(\tau_m+(i-1)\bar T_1)-L^{\infty}((\tau_m + i\bar T_1)-)
\right]\\
&\qquad+ \left[L^{\infty}(\tau_m+k\bar T_1)-L^{\infty}(T^{**}(\tau_m)-) \right]\\
&\ge 
\sum_{i=1}^{k}
\left[
L^{\infty}(\tau_m+(i-1)\bar T_1)-L^{\infty}((\tau_m + i\bar T_1)-)
\right]\,.
\end{align*}
We note that $k$ depends on $m$ and actually, it tends to infinity as $m\to\infty$. This allows us to choose $m$ large enough such that the number $k$ of subinterval of length $\bar T_1$ is sufficiently large. 
More precisely, there exists $m_0$ larger than $m_1$ such that $\delta\cdot k_0 \ge  1$, where $k_0:= k(m_0)$.  
Now, for $m=m_0$ assume that $L^{\infty}(\tau_{m_0}+(i-1)\bar T_1)-L^{\infty}((\tau_{m_0} + i\bar T_1)-)>\frac\delta2 { L^\infty(\tau_{m_0})}$ for every $i=1,\dots k_0$, then 
$$
\frac\delta2 k_0 { L^\infty(\tau_{m_0})}< \sum_{i=1}^{k_0}
\left[
L^\infty(\tau_{m_0}+(i-1)\bar T_1)-L^\infty((\tau_{m_0} + i\bar T_1)-)
\right]<\frac 12 { L^\infty(\tau_{m_0})}
$$ 
and this leads to a contradiction according to the choice of $k_0$. Therefore, there exists $j\in\{1,\dots,k_0\}$ such that 
\begin{equation}\label{S5-estimate-Linfty}
L^\infty(\tau_{m_0}+(j-1)\bar T_1)-L^\infty(\tau_{m_0}+j\bar T_1 -)\le \frac\delta2 { L^\infty(\tau_{m_0})}\,.
\end{equation}
Finally, since $\bar T_1> \bar T$, we can choose $t_0$ such that
$$
\tau_{m_0}+(j-1)\bar T_1 \le t_0 < t_0 + \bar T<  \tau_{m_0}+j\bar T_1
$$
so that $L^\infty$ is continuous at both $t_0$ and $t_0+\bar T$. Again by the monotonicity of $L^\infty$ one has that
\begin{equation*}
L^\infty(t_0)-L^\infty(t_0+\bar T)\le 
L^\infty(\tau_{m_0}+(j-1)\bar T_1)-L^\infty(\tau_{m_0}+j\bar T_1 -)
\end{equation*}
and hence, by means of \eqref{S5-estimate-Linfty}, we conclude that claim \eqref{claimTbarintervalnew} holds\,.

Next, having estimate~\eqref{claimTbarintervalnew}, we choose $\nu^*$ large enough such that 
$$
|L^\nu(t_0)-L^\infty(t_0)|+|L^\nu(t_0+\bar T)-L^\infty(t_0+\bar T)|<\frac\delta2 { L^\infty(\tau_{m_0})}
$$
for all $\nu\ge\nu^*$, since $L^\nu(t)\to L^\infty(t)$ as $\nu\to\infty$. Hence, 
\begin{equation}\label{claimTbarintervalnew2}
    L^\nu(t_0) - L^\nu(t_0+\bar T) \le \delta { L^\infty(\tau_{m_0})}\,\qquad \forall\,\nu\ge\nu^*\;.
\end{equation}
Let now $\nu_1^*\ge\nu^*$ such that $\MM\DT_\nu<1$ for all $\nu\ge\nu_1^*$. By~\eqref{hyp:Rtto0L} in Lemma~\ref{lem:rar-vanishes-for-large-t} and~\eqref{claimTbarintervalnew2}, we obtain that 
\begin{equation}\label{claimTbarintervalnew22}
R^\nu(t) \le  C e^{\MM\bar T}\, \delta { L^\infty(\tau_{m_0})}
\end{equation}
for all $t\in(t_0, t_0+\bar T)$ and $\nu\ge\nu^*_1$.
Moreover, we recall 
that $L^\infty (t_0 + \bar T) >\frac 12 L^\infty(\tau_{m_0})$ from \eqref{eq:ge-ge-g} and hence, we can find that 
\begin{align*}
    L^\nu (t_0 + \bar T) &\ge \frac 12L^\infty(\tau_{m_0}) - |L^\nu(t_0+\bar T)-L^\infty(t_0+\bar T)|\\
&    \ge \frac 12 (1-\delta)L^\infty(\tau_{m_0}) > \frac 14 L^\infty(\tau_{m_0}) \qquad\forall \,\, \nu\ge\nu_1^*\,,
\end{align*}
as long as $0<\delta< \frac12$. Therefore, by estimate~\eqref{claimTbarintervalnew22}, we have
\begin{align*}
S^\nu(t)& = L^\nu(t)-R^\nu(t)\\
&> \frac 14 L^\infty(\tau_{m_0}) - C e^{\MM\bar T}\,{ \delta}{ L^\infty(\tau_{m_0})}> \frac 15 L^\infty(\tau_{m_0})>0
\end{align*}
for $t\in(t_0,t_0+\bar T)$, $\nu\ge\nu_1^*$ and $\delta>0$ sufficiently small.
This immediately implies that we have the following lower bound 
\begin{equation}\label{eq:Snutimestepsadd2new}
    \sum_{t_0 <  t^n\le t_0+\bar T} M\DT S^\nu(t^n-)\ge \frac 15 L^\infty(\tau_{m_0}) \MM \bar T
\end{equation}
true for all $\nu\ge\nu_1^*$ noting that the lower bound is independent of $\nu$.

Now, lets investigate the variation of $R^\nu(t)$ in the interval $t\in(t_0,t_0+\bar T)$ in a similar argument that was used before. First, we write the expansion
\begin{equation*}
    R^\nu(t)=R^\nu(t_0)+\sum_{t_0<\tau<t} \left([\Delta R^\nu(\tau)]_+-[\Delta R^\nu(\tau)]_-\right)
\end{equation*}
where $[f]_+$ and $[f]_-$ denote the positive and negative part of $f$, respectively. 
Thanks to \eqref{claimTbarintervalnew22}, we find that 
\begin{align}\nonumber
\sum_{t_0<\tau<t} [\Delta R^\nu(\tau)]_+ &= R^\nu(t) - R^\nu(t_0) + \sum_{t_0<\tau<t} [\Delta R^\nu(\tau)]_- \\
&\le C e^{M\bar T}\, \delta L^\infty(\tau_{m_0}) + \sum_{t_0<\tau<t} [\Delta R^\nu(\tau)]_-\,, \label{eq:R-nu-t1barT22new}
\end{align}
for $\nu>\nu_1^*$  and $t\in(t_0,t_0+\bar T)$. To bound the negative variation of $R^\nu$, we list the three possible cases for which $[\Delta R^\nu(\tau)]_-$ is nonzero and these are:

\smallskip
(i) when there is a shock-rarefaction interaction of the same family resulting to a shock or a rarefaction of the same family. In view of the analysis of the term $(IV)$ in Lemma~\ref{lem:rar-vanishes-for-large-t}, it holds $[\Delta R^\nu(\tau)]_-\le C\,|\Delta L^\nu(\tau)|$ in this case, where $C$ is given at~\eqref{hyp:RK}; 

\smallskip
(ii) when a rarefaction front reaches the boundary $y=0$ or $y=\MM$. In this case, we have again $[\Delta R^\nu(\tau)]_-\le \,|\Delta L^\nu(\tau)|$ as mentioned in term $(II)$;

\smallskip
(iii) when a rarefaction front $\eps^-$ meets a time step at $\tau=t^n$. In this case, we have $[\Delta R^\nu(\tau)]_-\le \frac 1 2 \MM \DT |\eps^-|$. One can verify this from the analysis in term $(V)$ in Lemma~\ref{lem:rar-vanishes-for-large-t}. Summing over all time steps and using~\eqref{claimTbarintervalnew22}, we reach
$$
\sum_{
\genfrac{}{}{0pt}{}{t_0<t^n<t} {\eps^-\in\mathcal{R}}
} [\Delta R^\nu(t^n)]_-\le
\frac \MM 2  \DT \sum_{
\genfrac{}{}{0pt}{}{t_0<t^n<t} {}
}
|R^\nu(t^n-)| \, \le\, \frac \MM 2  \bar T C e^{\MM \bar T} \delta  L^\infty(\tau_{m_0})
$$
for all $t\in(t_0,t_0+\bar T)$ and $\nu\ge\nu^*_1$. 

In view now of the above three cases and~\eqref{claimTbarintervalnew2}, we have 
$$
\sum_{t_0<\tau<t} [\Delta R^\nu(\tau)]_- \le C \left( \frac \MM 2  \bar T  e^{\MM \bar T}+1\right) \delta  L^\infty(\tau_{m_0})
$$
for all $t\in(t_0,t_0+\bar T)$ and $\nu\ge\nu^*_1$. 
By means of the above, going back to \eqref{eq:R-nu-t1barT22new}, we infer that the sum of the positive part is also small of order $\delta$, i.e.
\begin{equation}\label{eq:DRpos2new}
\sum_{t_0<\tau<t} [\Delta R^\nu(\tau)]_+\le  C \left( \frac \MM 2  \bar T  e^{\MM \bar T}+e^{\MM \bar T}+1\right)\delta   L^\infty(\tau_{m_0})
\end{equation}
for all $t\in(t_0,t_0+\bar T)$ and $\nu\ge\nu^*_1$. 

In turn, we investigate the  cases  involved in the positive variation of for $R^\nu$, which are two:

(i) when there is a shock-shock interaction of the same family and hence, a reflected rarefaction arises after the interaction. In this case $[\Delta R^\nu(\tau)]_+\ge 0$;

(ii) when a shock front $\eps^-$ meets a time step at time $\tau=t^n$  and a rarefaction front $\eps_{refl}$ of the other family arises after the update. From \cite[Proposition 3.2]{AC_2021}, it holds $[\Delta R^\nu(t^n)]_+=|\eps_{refl}|\ge \MM\DT c_1(q) |\eps^-|$, where $c_1(q)= (1+\cosh(q))^{-1}$.

Taking into account only case (ii) that involves the time steps, we get
$$
\sum_{t_0<\tau<t} [\Delta R^\nu(\tau)]_+\ge \sum_{
\genfrac{}{}{0pt}{}{t_0<t^n<t} {\eps^-\in\mathcal{S}}
} [\Delta R^\nu(t^n)]_+\ge  c_1(q) \sum_{t_0<t^n<t}  \MM\DT S^\nu(t^n-)\,.
$$
Combining with~\eqref{eq:Snutimestepsadd2new} for $t\in(t_0,t_0+\bar T)$ and $\nu\ge\nu^*_1$, we arrive at
$$
\sum_{t_0<\tau<t} [\Delta R^\nu(\tau)]_+\ge c_1(q)\frac 15 L^\infty(\tau_{m_0}) \MM \bar T>0\,,
$$
which contradicts~\eqref{eq:DRpos2new} since $\delta$ can be chosen arbitrarily small. Thus, $T^*$ is finite and the proof is complete.
\end{proof}

\subsection{Proof of Theorem~\ref{Th-2-unconditional}}\label{Subsec:5.1}
Here we conclude the proof of Theorem~\ref{Th-2-unconditional}, based on the estimates proved in the previous sections. 

First, we return to the notation of Step 1 of Section~\ref{S3} and recall that $(\tilde\rho(x',t'), \tilde\mm(x',t'))$ is the solution to~\eqref{eq:system_Eulerian_M-M1}, or equivalently, ~\eqref{eq:system_Eulerian_M-M1-K=1} with kernel $1$, while $(\rho, \mm)$ is the one to~\eqref{eq:system_Eulerian_M-M1-K} with constant kernel $K>0$.
Then, from Step $2$ of Section~\ref{S3}, we have the existence of an entropy weak solution $(\tilde\rho, \tilde\mm)$ with concentration along the free boundaries $\tilde a(t')$ and $\tilde b(t')$  on the $(x',t')$ plane. The solution $(\tilde\rho, \tilde\mm)$ conserves mass and momentum that are denoted by $\widetilde\MM$ and $\widetilde\MM_1$, respectively, and there exists a positive constant $\rho_{inf}>0$ such that the density is uniformly positive within its support, i.e. $\essinf_{x'\in \tilde I(t')} \tilde\rho(\cdot,t')\ge \rho_{inf}$. Moreover, the solution satisfies the results on the decay of waves obtained in Sections~\ref{S2.1} and~\ref{S5}. 

From 
Lemma~\ref{S5Prop 5.1} and by iteration, estimate~\eqref{T*star} would lead to
\begin{equation}\nonumber 
    L^\infty(t') \le L^\infty(k\, T^* )
    \le \frac 1{2^k}\,L^\infty(0),\qquad \forall~ t' \in [k T^*, (k+1) T^*) 
\end{equation} 
and hence 
\begin{equation}\nonumber
    L^\infty(t')\le 2  \, {2^{-\frac { t'}  {T^*}}}\,\,L^\infty(0) = C_1 {e}^{-C_2 t'}
\end{equation} 
for appropriate constants $C_1$, $C_2>0$ independent of $t'$. Using similar argument now as in Proposition~\ref{prop:unif-conv-of-L}, we can prove that $L^\nu(t')$ converges to 0 exponentially fast as well in time and also uniformly in $\nu$ and Lemma~\ref{lem:3.2} is proven.

As in \cite[(5.31),(5.32)]{AC_2021}, it follows that the total variation of the approximate sequence decays to zero exponentially fast as time tends to infinity. Therefore, passing to the limit as $\nu\to\infty$ and using the liminf property of the total variation, we get the following exponential bounds:
\begin{align*}
\tv\left\{\tilde \rho(t');\,(\tilde a(t'),\tilde b(t'))\right\} \,,\quad
\tv\left\{\tilde \vv(t');\,(\tilde a(t'),\tilde b(t'))\right\}  
\le C_2'e^{- C_1'  t'},\,\qquad \forall\, t' > 0 
\end{align*}
and also, following \cite[Section 5.3]{AC_2021}, there exists a constant $\rho_\infty>0$ such that
\begin{align*}
\esssup_{x\in (\tilde a(t'), \tilde b(t'))} |\tilde\rho(x,t)- \rho_\infty|
\le C_2'e^{- C_1'  t'},\,\qquad \forall\, t'> 0
\end{align*}
as well as
\begin{align*}
\esssup_{x\in (\tilde a(t'), \tilde b(t'))} |\tilde\vv(x,t)|
\le C_2'e^{- C_1'  t'},\,\qquad \forall\, t'> 0
\end{align*}
for appropriate constants $C_1'$, $C_2'>0$ independent of $t'$. Here we recall that we set 
$\MM_1=\bar \vv =0$ in Step~1, Section~\ref{S3}.

In addition, the free boundaries satisfy
$$
\lim_{t'\to\infty} \tilde a(t')=\tilde a_\infty\doteq 
\tilde a_0+\int_0^\infty \tilde v(0+,s')\,ds'  \,,
\qquad
\lim_{t'\to\infty} \tilde b(t')=\tilde a_\infty+ \frac{\widetilde \MM}{\rho_\infty} \;.
$$
Thus, $(\tilde\rho,\tilde\mm)$ admits time-asymptotic flocking exponentially fast.

At last, according to Step 1 in Section~\ref{S3}, we translate the variables $x'$ and $\tilde\vv(x',t')$ and then  rescale the system backwards using $\Phi_\lambda^{-1}$, with $\lambda=\sqrt{K}$, to go back to the original $(x,t)$ variables with $\MM_1$ and $\bar \vv$ provided by \eqref{cons-of-momentum} and \eqref{def:vbar} respectively. In this way, we obtain the BV function $(\rho(x,t),\mm(x,t))$ that is now the entropy weak solution to~\eqref{eq:system_Eulerian_M-M1-K} with concentration along the locally Lipschitz curves 
$$ 
a(t)=\frac{1}{\lambda } \left(\tilde a(t') + \bar \vv t'\right) = \frac{1}{\lambda } \tilde a(\lambda t) + \bar \vv t,\qquad 
b(t) = \frac{1}{\lambda } \tilde b(\lambda t) + \bar \vv t\,,\qquad t'=\lambda t
$$ 
and satisfies $\essinf_{x\in  I(t)} \rho(\cdot,t)\ge \rho_{inf}$. From above,
$$
a(t) - \bar\vv t\longrightarrow a_\infty\doteq \frac{a_\infty}{\lambda} =a_0+\int_0^\infty\vv(a(s)+,s) ds \,,\qquad b(t)- \bar\vv t\longrightarrow  a_\infty+ \frac{ \MM}{\rho_\infty}\;,
$$ 
as $t\to+\infty$. Since $(\rho(x,t),\mm(x,t))$  conserves mass $\MM=\frac{1}{\lambda} \tilde \MM$ and momentum $\MM_1=\frac{1}{\lambda} \tilde \MM_1$, it immediately satisfies the integral identities~\eqref{S1:rho-eq-phi2},~\eqref{S1:m-eq-phi2-pressure} and~\eqref{entropy-cond_rho-m} with $M_1(t)=\MM_1$. Now the support $I(t)=[a(t),b(t)]$ of $(\rho,\mm)$ is uniformly bounded
$$
b(t)-a(t)=\frac{1}{\lambda}\left( \tilde b(t)-\tilde a(t)\right)\le \frac{1} {\lambda \rho_{inf} }\tilde \MM=
\frac{1} { \rho_{inf}} \MM
$$
for all $t>0$, and 
\begin{equation*}
    \esssup_{x\in (a(t), b(t))} |\rho(x,t)- \rho_\infty|\,,\qquad \esssup_{x\in (a(t), b(t))} |\vv(x,t) - \bar \vv|
\le C_2'e^{- C_1'  t},\,\qquad \forall\, t> 0\,.
\end{equation*}

The proof of the theorem is now complete.
\qed


\begin{thebibliography}{20}
\bibitem{ABCD_JEE_2015} D.~Amadori, P.~Baiti, A.~Corli and E.~Dal~Santo,
\newblock Global weak solutions for a model of two-phase flow with a single interface,
\newblock {\em J. Evol. Equ.} {\bf 15} (2015), no. 3, 699--726

\bibitem{AC_2021} D.~Amadori, C.~Christoforou,
\newblock $BV$ solutions for a hydrodynamic model of flocking--type with all-to-all interaction kernel,
\newblock {\it Math. Models Methods Appl. Sci.} {\bf 11}(32) (2022), 2295--2357

\bibitem{AC_SIMA_2008} D.~Amadori and A.~Corli,
\newblock On a model of multiphase flow,
\newblock {\em SIAM J. Math. Anal.} {\bf 40} (2008), no. 1, 134--166

\bibitem{AmadoriGuerra01} D.~Amadori and G.~Guerra,
\newblock Global {$BV$} solutions and relaxation limit for a system of conservation laws,
\newblock {\it Proc. Roy. Soc. Edinburgh Sect. A} {\bf 131} (2001), no. 1,  1--26   

\bibitem{BCC-2011} F.~Bolley, J.A.~Ca\~nizo, and J.A.~Carrillo, 
\newblock Stochastic mean-field limit: Non-lipschitz forces and swarming,
\newblock {\it Math. Mod. Meth. Appl. Sci.} {\bf 21} (2011) 2179--2210

\bibitem{Choi2019} {Y.-P. Choi},
\newblock The global Cauchy problem for compressible Euler equations with a nonlocal dissipation,
{\it Math. Models Methods Appl. Sci.} {\bf 29} (2019) 185--207

\bibitem{CuS1} F. Cucker and S. Smale, 
\newblock Emergent behavior in flocks,
\newblock {\em IEEE Transactions on automatic control} \textbf{ 52} (2007), no. 5, 852--862

\bibitem{CFTV-2010} J.A. Carrillo, M.~Fornasier, G.~Toscani and F. Vecil,
\newblock Particle, kinetic, and hydrodynamic models of swarming, 
\newblock In {\it Mathematical modeling of collective behavior in socio-economic and life sciences}, Model. Simul. Sci. Eng. Technol. (2010) 297--336

\bibitem{Dafermosbook} C.M. Dafermos,
\newblock {\it Hyperbolic Conservation Laws in Continuum Physics},
\newblock Fourth Edition, Grundlehren der Mathematischen Wissenschaften {\bf 325}, Springer Verlag, Berlin, 2016

\bibitem{Dafermos_frictional} C.M. Dafermos,
\newblock A system of hyperbolic conservation laws with frictional damping,
\newblock {\it Z. Angew. Math. Phys.} \textbf{46} (1995), Special Issue, S294--S307

\bibitem{Frid96} H.~Frid,
\newblock Initial-boundary value problems for conservation laws, 
\newblock {\it J. Differential Equations} {\bf 128} (1996), 1--45

\bibitem{HT08} S.-Y. Ha and E. Tadmor, 
\newblock From particle to kinetic and hydrodynamic descriptions of flocking,
\newblock {\em Kinetic and Related Model} \textbf{1} (2008), no.~3, 415--435


\bibitem{KMT13} T. Karper, A. Mellet and K. Trivisa, 
\newblock Existence of weak solutions to kinetic flocking models, 
\newblock {\it SIAM J. Math. Anal.} {\bf 45} (2013), 215--243

\bibitem{KMT15} T. Karper, A. Mellet and K. Trivisa, 
\newblock Hydrodynamic limit of the kinetic Cucker-Smale model,
\newblock {\it Math. Models Methods Appl. Sci.} \textbf{25} (2015), 131--163

\bibitem{LuoNatYan} T.~Luo, R.~Natalini and T.~Yang, 
\newblock Global {$BV$} solutions to a $p$--system with relaxation, 
\newblock {\it J. Differential Equations} {\bf 162} (2000), 174--198

\bibitem{MT11} S.~Motsch and E.~Tadmor, 
\newblock A new model for self-organized dynamics and its flocking behavior,
\newblock {\it  J. Stat. Phys.} {\bf 141} (2011), 923--947

\bibitem{Nishida68} T. Nishida, 
\newblock Global solution for an initial boundary value problem of a quasilinear hyperbolic system,  
\newblock {\it Proc. Japan Acad.} {\bf 44} (1968), 642--646

\bibitem{Shv2021} R. Shvydkoy, 
\newblock {\it Dynamics and Analysis of Alignment Models of Collective Behavior},
\newblock Ne\v{c}as Center Series, Birkh\"auser, 2021

\bibitem{Ta2023} E. Tadmor,
\newblock Swarming: hydrodynamic alignment with pressure,
\newblock Bull. Amer. Math. Soc. {\bf 60} (2023), 285--325

\bibitem{W87} D.H. Wagner,
\newblock Equivalence of the Euler and Lagrangian equations of gas dynamics for weak solutions.
\newblock {\it J. Differential Equations} {\bf 68} (1987), 118--136

\end{thebibliography}
\end{document}